\documentclass[12pt,a4]{amsart}
\usepackage[a4paper, left=28mm, right=28mm, top=28mm, bottom=34mm]{geometry}
\usepackage[all]{xy}
\usepackage{enumerate}
\usepackage{amsmath}
\usepackage{amssymb}
\usepackage{amsthm}
\usepackage{mathrsfs}
\newtheorem{thm}{Theorem}[section]
\newtheorem{lem}[thm]{Lemma}
\newtheorem{cor}[thm]{Corollary}
\newtheorem{prop}[thm]{Proposition}

\theoremstyle{definition}

\newtheorem{rem}[thm]{Remark}

\newtheorem{defn}[thm]{Definition}

\newtheorem{ex}[thm]{Example}

\numberwithin{equation}{section}

\def\B{{\mathbb B}}
\def\BB{{\mathfrak B}}

\def\F{{\mathbb F}}

\def\Q{{\mathbb Q}}
\def\R{{\mathbb R}}
\def\Z{{\mathbb Z}}
\def\C{{\mathbb C}}
\def\O{{\mathscr O}}

\def\cU{{\mathcal U}}

\def\L{{\mathcal L}}
\def\M{{\mathcal M}}

\def\aa{{\mathfrak a}}

\def\oo{{\mathfrak o}}
\def\OO{{\mathfrak O}}

\def\wL{\widetilde{L}}

\def\CC{{\mathfrak C}}

\def\pp{{\mathfrak p}}
\def\PP{{\mathfrak P}}

\def\qq{{\mathfrak q}}
\def\wV{\widetilde{V}}

\def\wR{\widetilde{\mathbb R}}
\def\val{{\mathrm{val}}}
\def\res{{\mathrm{res}}}

\def\Cl{\mathop{\mathrm{Cl}}\nolimits}

\def\Fin{\mathop{\mathrm{Fin}}\nolimits}

\def\ccSb{\mathop{\mathcal{S}^u_{0,b}}\nolimits}

\def\cF{\mathop{\mathcal{F}}\nolimits}
\def\wf{\mathop{\tilde{f}}\nolimits}
\def\wg{\mathop{\tilde{g}}\nolimits}

\def\>'{\mathop{>'}}

\def\cI{\mathop{\mathcal{I}}\nolimits}

\def\Loc{\mathop{\mathcal{L}oc}\nolimits}

\def\cM{\mathop{\mathcal{M}}\nolimits}
\def\cN{\mathop{\mathcal{N}}\nolimits}
\def\ccSu{\mathop{\mathcal{S}^u_0}\nolimits}
\def\cSub{\mathop{\mathcal{S}^u_b}\nolimits}
\def\cS{\mathop{\mathcal{S}}\nolimits}

\def\ccS{\mathop{\mathcal{S}_0}\nolimits}
\def\cSu{\mathop{\mathcal{S}^u}\nolimits}
\def\AbMon{\mathop{\mathcal{A}b\mathcal{M}on}\nolimits}

\def\ccSr{\mathop{\mathcal{S}^u_{0,\mathrm{ref}}}\nolimits}

\def\cP{\mathop{\mathcal{P}}\nolimits}
\def\cF{\mathop{\mathcal{F}}\nolimits}
\def\wF{\mathop{\widetilde{F}}\nolimits}

\def\Frac{\mathop{\mathrm{Frac}}\nolimits}
\def\mSpec{\mathop{\mathrm{m \text{\rm -}Spec}}\nolimits}
\def\Ram{\mathop{\mathrm{R}^{\infty}_{L/K}}\nolimits}
\def\IF{\mathop{\mathcal{IF}_{\text{\rm fin}}}\nolimits}

\def\Spec{\mathop{\mathrm{Spec}}\nolimits}

\def\Dom{\mathop{\mathcal{D}om}\nolimits}

\def\Gal{\mathop{\mathrm{Gal}}\nolimits}

\def\Hom{\mathop{\mathrm{Hom}}\nolimits}

\def\Lim{\mathop{\mathrm{Lim}}\nolimits}

\def\Max{\mathop{\mathrm{Max}}\nolimits}

\def\Reg{\mathop{\mathrm{Reg}}\nolimits}

\def\Pic{\mathop{\mathrm{Pic}}\nolimits}
\def\Ker{\mathop{\mathrm{Ker}}\nolimits}

\def\Max{\mathop{\rm Max}}

\def\B{\mathbb{B}}

\begin{document}

\title[On ideal class semigroups]
{On infinite extensions of Dedekind domains, 
upper semicontinuous functions 
and the ideal class semigroups}

\author{Tatsuya Ohshita}

\address{Faculty of Science, Ehime University 2--5, Bunkyo-cho, Matsuyama-shi, Ehime 790--8577, Japan}
\email{ohshita.tatsuya.nz@ehime-u.ac.jp}

\date{\today}

\begin{abstract}
In this article, we study the monoid of  fractional ideals 
and the ideal class semigroup of an arbitrary given one dimensional 
normal domain $\OO$ obtained by an infinite integral extension of 
a Dedekind domain. 
We introduce a notion of ``upper semicontinuous functions" 
whose domain is the maximal spectrum of $\OO$ 
equipped with 
a certain topology, and whose codomain 
is a certain totally ordered monoid containing $\R$.  
We construct an isomorphism 
between a monoid consisting of such upper semicontinuous functions  
satisfying certain conditions 
and the monoid of fractional ideals of $\OO$. 
This result can be regarded as 
a generalization of the theory of 
prime ideal factorization for Dedekind domains. 
By using such isomorphism, 
we study the Galois-monoid structure of 
the ideal class semigroup of $\OO$.
\end{abstract}

\maketitle

\section{Introduction}\label{secintro}

The ideal class semigroup of a commutative ring, 
which is a quotient of 
the monoid of fractional ideals 
by principal fractional ideals,  has been 
studied by some authors,  
for instance, by 
Bazzoni--Salce (\cite{BS}), 
Bazzoni (\cite{Ba1}--\cite{Ba3}),
Zanardo--Zannier (\cite{ZZ}) and 
Konomi--Morisawa (\cite{KM}). 
In particular, Konomi and Morisawa studied 
the ideal class semigroups
in number theoretic, and Galois theoretic setting: 
In \cite{KM}, they gave a description of the 
structure of the ideal class semigroups 
of the ring of integers in 
the cyclotomic $\Z_p$-extension of $\Q$ 
by using the Picard group, 
which is a group consisting of the classes 
of invertible fractional ideals.
Motivated by their work, 
in this article, 
we study the monoids of fractional ideals 
and the ideal class semigroups of 
arbitrary one dimensional normal domains 
which are obtained by
infinite extension of Dedekind domains.

Let us fix our notation. 
Throughout  this article, 
let $\oo$ be a Dedekind domain which is not a field.  
We denote by $K$ the fractional field of $\oo$. 
Let $L/K$ be an algebraic extension of fields.
We denote by $\OO$ the integral closure of $\oo$ in $L$.
In our setting, 
the extension degree of $L$ over $K$
may be infinite. 
So, the ring $\OO$ is not  Noetherian in general.
In particular, 
the ring $\OO$ may not be a Dedekind domain. 
%Related to this fact, 
%the monoid $\cF(\OO)$ of fractional ideals of $\OO$
%may not be a group in our situation.

Let us recall the monoid of fractional ideals. 
We call a non-zero $\OO$-submodule $I$ of $L$ 
\textit{a fractional ideal of $\OO$} if  
there exists an element $d \in L^\times$ 
such that $d I \subseteq \OO$. 
We say that a fractional ideal 
$I$ of $\OO$ is \textit{an invertible fractional ideal} 
(or \textit{an invertible ideal} for short)
if there exists a fractional ideal $J$ of $\OO$ 
satisfying $IJ=\OO$. 
In particular, \textit{principal fractional ideals}, 
which are fractional ideals 
of the form $a \OO$ for some $a \in L^\times$, are invertible.
Note that there may exist a fractional ideal of $\OO$
which is not invertible 
since $\OO$ may not be a Dedekind domain in our setting. 
We denote the set of all fractional 
(resp.\ invertible fractional and principal fractional)
ideals of $\OO$ by $\cF (\OO)$ 
(resp.\ $\cI (\OO)$ and $\cP (\OO)$). 
By the multiplication of ideals, 
the sets $\cF (\OO)$, $\cI (\OO)$ and $\cP (\OO)$
become abelian monoid. 
Moreover, the monoids 
$\cI (\OO)$ and $\cP (\OO)$ are groups.

For a monoid $M$, we say that 
an element $a \in M$ is {\it von Neumann regular} 
(or {\it regular} for short) 
if there exists an element $b \in M$ satisfying 
$a=a ba$. 
\if0
(Note that regular elements of a monoid 
are usually called invertible elements. 
However, ``invertible ideals" 
and ``invertible elements of $\cF (\OO)$" 
are {\it different} notions. Indeed, all invertible ideals 
are invertible elements of $\cF (\OO)$, 
but the converse does not hold in general. 
This is very confusing, so we prefer the terminology ``regular"
in our context.)
\fi
We denote by $\Reg M$ the subset of $M$
consisting of all the regular elements. 
Since $\cF (\OO)$ is an abelian monoid, 
the set $\Reg \cF (\OO)$ becomes a submonoid
of $\cF (\OO)$. 
By definition, all invertible fractional ideals are 
contained in $\Reg \cF (\OO)$. 
%However, in general, there may exist
%an element of $\Reg \cF (\OO)$ not contained in 
%$\cI (\OO)$. 

In this article, we study 
the ideal class semigroup $\Cl (\OO)$ of $\OO$
and its regular part $\Reg \Cl (\OO)$ 
via the following topological, 
and valuation theoretic approach. 
Note that by  prime ideal factorization, 
the fractional ideals of a Dedekind domain are 
parameterized by discrete valuations corresponding to 
non-zero prime ideals.
In this article, in order to study fractional ideals of $\OO$, 
we shall generalize this theory 
by using topological terminology. 
We shall introduce a totally ordered monoid 
$\wR = \R \times \{ 0,1 \}$, 
and define the notion of 
$\wR$-valued {\it upper $\wR$-semicontinuous} functions
on the maximal spectrum $\mSpec \OO$ 
equipped with a certain new topology 
called {\it Krull-like topology}. (See Definition \ref{defntop}.)
By using an abelian monoid $\ccSb (\OO ; \wV)$ 
consisting of such functions 
satisfying certain conditions, 
we shall construct a ``natural" isomorphism
\[
\wF \colon
\cF (\OO) \xrightarrow{\ \simeq \ }  
\ccSb (\OO ; \wV).
\]
(For details, see Theorem 
\ref{thmmain1}. For the precise meaning of 
the ``naturality" of $\wF$, see Remark \ref{remnatural}.)
Note that if $\OO$ is a Dedekind domain, then 
$\wF$ coincides with 
the prime ideal factorization 
\[
\cF (\OO) \xrightarrow{\ \simeq \ }  
\ccSb (\OO ; \wV) \simeq 
\bigoplus_{\PP \in \mSpec} \Z ;\ 
\longmapsto ( 
\text{the multiplicity of $\PP$ in $I$ })_\PP. 
\]
(See Example \ref{exded}.) 
So, our result can be regarded as
a generalization of the theory of 
prime ideal factorization for Dedekind domains. 
Moreover, we can determine the images of 
$\Reg \cF (\OO)$ and $\cI (\OO)$
by $\wF$ respectively.
So, via the isomorphism $\wF$, 
wecan describe  the quotient
of the ideal class semigroup $\Cl (\OO):= \cF (\OO)/\cP(\OO)$ 
by the Picard group $\Pic (\OO):= \cI(\OO) /\cP(\OO)$, and 
it von Neumann regular part. 
(See Corollary \ref{corthmmain}.) 
In particular, we can show that 
if $\OO$ has {\it the  
finite character property}, 
namely if every non-zero element of $\OO$ 
is contained in at most a finite number of maximal ideals, 
then $\wF$ induces an isomorphism  
\[
\Cl (\OO)/\Pic (\OO) \simeq \bigoplus_{\PP}  
\left( \{0 \} \sqcup \R/v_\PP (L^\times) \right),
\]
where $\PP$ runs through all maximal ideals of $\OO$ 
such that the image of an additive valuation
$v_\PP \colon L^\times \longrightarrow \Q$ 
corresponding to $\PP$ is dense in $\R$. 
(See Corollary \ref{corClifford1}.)
For examples for the cases when 
$\OO$ does  not have the  
finite character property, 
see \S \ref{secEx}. 

If $L/K$ is Galois, then 
the Galois group $G:=\Gal(L/K)$
naturally acts on the monoids $\cF(\OO)$ 
and $\ccSb (\OO ; \wV)$. 
We can show that the isomorphism $\wF$ 
preserves the action of $G$. 
So, by using the isomorphism $\wF$, 
we can study the Galois-monoid structure of 
$\Cl (\OO)/\Pic (\OO)$.
(For details, see Corollary 
\ref{corcorthmmain}, Corollary \ref{corClifford2} 
and examples in \S \ref{secEx}.)
Our results can be regarded as 
a generalization of ``modulo the Picard group"
version of the results  of Konomi and Morisawa \cite{KM}. 
Note that in order to describe not only $\Cl (\OO) /\Pic (\OO)$
but also $\Cl (\OO)$ by using the isomorphism $\wF$,
we need to know some additional information on 
the image of the principal fractional ideals 
by $\wF$.
(For instance, see Remark \ref{remKM}, that explains 
the result by Konomi and Morisawa \cite{KM},
and examples in \S \ref{secEx}.)

In \S \ref{secmainthm}, 
we state the precise statement of our results, 
and give some remarks on them.
In \S \ref{secmonoid}, we study the basic properties of 
upper $\wR$-semicontinuous functions. 
In \S \ref{ssclasstop}, we study 
the topological space 
$\mSpec \OO$ equipped with the Krull-like topology. 
In \S \ref{secpf}, we prove our main theorem, 
namely Theorem \ref{thmmain1}, 
the main part of whose assertion 
is the existence of the isomorphism $\wF$. 
In \S \ref{secEx}, we observe two examples 
of the ideal class semigroups described by 
the isomorphism $\wF$. 
The first one is the Tate module covering of an elliptic curve over $\C$, 
and the second one 
is the maximal real abelian extension of $\Q$.

\section*{Acknowledgment}

This work is inspired by  
Yutaka Konomi and Takayuki Morisawa's work \cite{KM}.
This work is also motivated by conversation between 
Takashi Hara and Yasuhiro Ishitsuka. 
The author would like to thank them 
for the fruitful conversation. 
The author is grateful to Kenji Sakugawa 
for his helpful comments.
This work is supported by 
JSPS KAKENHI Grant Number 26800011.

\section{Main results}\label{secmainthm}

Let $\oo$, $L/K$ and $\OO$ be as in \S \ref{secintro}.
In this section, we shall state our main results. 
In \S \ref{ssparamideal}, we state the precise statement of 
the existence of the isomorphism $\wF$
referred in \S \ref{secintro},  
which becomes 
a key of our paper.
In \S \ref{ssisg}, by using the isomorphism $\wF$, 
we study the structure of the quotient of 
the ideal class semigroup $\Cl (\OO)=\cF (\OO)/ \cP (\OO)$ 
by the Picard group $\Pic (\OO)=\cI (\OO)/ \cP (\OO)$. 
In \S \ref{ssGal}, we observe the isomorphism $\wF$
under the assumption that $L/K$ is Galois, 
and study the Galois-monoid structure of 
the ideal class semigroup $\Cl (\OO)$.

\subsection{Parameterization of ideals}\label{ssparamideal}\label{ssmainthm}
Here, let us state our main results on 
the parameterization of fractional ideals of $\OO$
by using ``upper $\wR$-semicontinuous functions" 
on $\mSpec \OO$. 
We need to set some notation.

Let $R$ be a commutative ring. 
We denote by $\mSpec R$
the set of all maximal ideals of $R$. 
For a set $S$, 
we denote by $\Fin (S)$  the collection of all finite subsets of $S$. 
For each $A \in \Fin (R)$, 
we define $\cU (A)=\cU_R (A):=\{ \pp 
\in \mSpec R \mathrel{\vert} A \subseteq \pp \}$. 
Then, we can define a topology on $\mSpec R$ called 
{\it Krull-like topology} in our article,  
whose open base is given by the collection 
$\BB (R):= \{ \cU(A) \mathrel{\vert} 
A \in \Fin (R) \}$. 
(For details, see \S \ref{ssclasstop}.)

In order to state our main results, 
let us briefly introduce the some notation related to 
``upper $\wR$-semicontinuous functions" on $\mSpec \OO$
which parameterize fractional ideals in our main results.
%(Here, we just introduce only the notation.
%For the definition and basic properties of the following notion, 
%see \S \ref{secmonoid}, later.)

First, let us introduce the totally ordered monoid $\wR$
which becomes the codomain of the functions. 
Put $\wR:= \R \times \{ 0,1 \}$.
By the lexicographic order $\le$,  we regard $\wR$ 
as a totally ordered set. 
We define the sum of $(\alpha ,s), (\beta ,t) \in \wR$ by 
$(\alpha ,s) + (\beta ,t):= 
(\alpha + \beta , \max\{ s,t \})$.
Then, the triple $(\R, \le , +)$ is 
a totally ordered abelian monoid. 
(For details, see Definition \ref{deftotordabmonoid} 
and Lemma \ref{lemtptom}.)

Take any $\PP \in \mSpec \OO$. 
Here, let us define a submonoid $\wV_\PP$ of $\wR$. 
We can take a unique additive valuation 
$v_{\PP} \colon L^\times \longrightarrow \Q$
corresponding to $\PP$ 
which is normalized by $v_{\PP} (K^\times)=\Z$.
(See Definition \ref{defval}.) 
For convenience, we put $v_\PP(0)=+ \infty$.
We define $V_\PP$ to be the closure of 
$v_\PP (L ^\times)$ in $\R$. 
Note that we have  $V_\PP=v_\PP(L^\times)$ 
(resp.\ $V_\PP=\R$) if $v_\PP (L ^\times)$ is 
discrete (resp.\ not discrete) in $\R$. 
We define a subset $\val_\pp$ of $\wR$
by $\val_\PP := 
\{ (\alpha ,0) \mathrel{\vert} \alpha \in v_\PP(L^\times)\}$, 
and put 
\[
\wV_\PP := 
\begin{cases}
\val_\PP 
\sqcup \{ (\beta ,1) \mathrel{\vert} 
\beta \in \R \} & (\text{if $V_\PP=\R$}), \\
\val_\PP & 
(\text{if $V_\PP=v_\PP(L^\times)$}).
\end{cases}
\]
Note that $\val_\PP$ and $\wV_\PP$ are submonoid of $\wR$, 
and moreover $\val_\pp$ becomes a group.
We define an involution $\iota$ on $\wR$ by 
$\iota (\alpha ,s):=(-\alpha ,s)$
for each $(\alpha ,s) \in \wR$. 
The submonoids $\wV_\PP$ and $\val_\PP$
are stable under the action of $\iota$.

Let $X$ be a topological space $X$, and  
$M$ a submonoid of $\wR$. 
We say that a function $f \colon X \longrightarrow \R$ 
is {\it upper $\wR$-semicontinuous} 
if and only if the set $f^{-1} (\wR_{<a} )$ is open 
for any $a \in \wR$. 
We denote by 
$\ccSb (X;M)$ the abelian monoid of all   
upper $\wR$-semicontinuous, 
bounded, and compactly supported functions on $X$ 
whose image is contained in $M$. 
(For details, see \S \ref{secmonoid}, in particular 
Definition \ref{defbdduppersemicont} 
and Lemma \ref{semicontmon}.) 
If  $M$ is stable under the action of $\iota$, 
then we can define a $G$-stable 
submonoid $\ccSr(X; M)$ of 
$\ccSb(X; M)$ by 
\(
\ccSr(X; M):= \{ 
f \in \ccSb(X; M) \mathrel{\vert} 
\iota \circ f \in \ccSb(X; M)
\}
\). 
Then, we define 
\begin{align*}
\ccSb (\OO;\wV):&=\{ \wf \in \ccSb (\mSpec \OO ;\wR)
\mathrel{\vert} \wf (\PP) \in \wV_\PP 
\ \text{for any $\PP \in \mSpec \OO$} \}, \\
\ccSr (\OO ;\wV):&=\{ \wf \in \ccSr (\mSpec \OO ;\wR)
\mathrel{\vert} \wf (\PP) \in \wV_\PP 
\ \text{for any $\PP \in \mSpec \OO$} \}, \\
\ccSr (\OO ;\val):&=\{ \wf \in \ccSr (\mSpec \OO ;\wR)
\mathrel{\vert} \wf (\PP) \in \val_\PP 
\ \text{for any $\PP \in \mSpec \OO$} \}.
\end{align*}
Note that $\ccSr (\OO ;\val)$ is a group whose elements are 
locally constant functions. (See Corollary \ref{corlocconst}.)
By definition, 
the group $\ccSr( \OO ;\val)$  
is  a submonoid of $\ccSr( \OO ;\wV)$ (and hence also 
that of $\ccSb ( \OO ;\wV)$).

Take any $I \in \cF(\OO)$. Here,  we define 
a function $\wf_I \colon \mSpec \OO \longrightarrow \wR$ 
as follows.  
First,  we define a function 
$f_{I}\colon \mSpec \OO \longrightarrow \R$ by 
\(
f_{I}(\PP):= \inf \{v_{\PP}(x) 
\mathrel{\vert} x \in I \}
\) 
for each $\PP \in \mSpec \OO$.
Let $\Max (I)$ be the subset of $\mSpec \OO$
consisting of all the elements $\PP$ such that 
the function $v_{\PP}\vert_{I} \colon I \longrightarrow \R$ 
has a minimum value, namely 
there exists an element $x \in I$ satisfying
\(
f_{I}(\PP) = v_{\PP} (x)
\). 
Note that if $\PP \in \Max (I)$, 
then we have $f_{I}(\PP) \in v_\PP(L^\times)$. 
We define a function 
$\wf_{I}\colon \mSpec \OO \longrightarrow \wR$ by 
\[
\wf_{I} (\PP):= \begin{cases}
(f_{I}(\PP), 0) & (\text{if $\PP \in \Max (I)$}), \\
(f_{I}(\PP), 1) & (\text{if $\PP \notin \Max (I)$}).
\end{cases}
\]
In \S \ref{sswF}, 
we shall prove that 
$\wf_I $ belongs to $\ccSb (\OO;\wV)$.
(See Lemma \ref{lemfIsc}.)
The main result on the parameterization of 
fractional ideals of $\OO$ is as follows.

\begin{thm}\label{thmmain1}
The following hold. 
\begin{enumerate}[{\rm (i)}]
\item We have a monoid isomorphism
\[
\wF=\wF_{\oo,L} \colon \cF(\OO) \xrightarrow{\ \simeq \ } 
\ccSb(\OO ; \wV);\ I \longmapsto \wf_I.
\]  
\item The image of $\Reg \cF (\OO)$ by $\wF$ coincides with 
$\ccSr(\OO; \wV)$.
\item The image of $\cI (\OO)$ by $\wF$ coincides with 
$\ccSr(\OO ; \val)$. 
\end{enumerate}
\end{thm}

\begin{rem}[Normalization of additive valuations]
Let $\PP \in \mSpec \OO$ be any element. 
Recall that we have 
normalized the additive valuation $v_\PP$ by
$v_{\PP}(K^\times) =\Z$. 
Here, we explain that 
this normalization is a key of 
the upper $\wR$-semicontinuity of 
the functions $\wf_I$.
Let $I \in \cF (\OO)$ be any element, and  
take any $x \in I \cap M$. 
Then, for any $\PP' \in \mSpec \OO$ 
satisfying $\PP' \cap K(x)=\PP \cap K(x)$,  
by the normalization 
$v_{\PP'}(K^\times) =\Z$,  
we deduce 
that the restriction of $v_{\PP'}$ on $K(x)$ 
coincides with that of $v_\PP$, 
and obtain the inequality 
$f_{I}(\PP') \le v_{\PP'}(x)=v_{\PP} (x)$.
This inequality becomes 
an essence of the upper $\wR$-semicontinuity of 
the function $\wf_I$.
(See the proof of Lemma \ref{lemfIsc}.)
\end{rem}

\begin{rem}[Discrete cases]\label{remfindex}
Here, we observe the monoids of functions appearing 
in Theorem \ref{thmmain1}
in the special cases when $\mSpec \OO$ is discrete. 
In our setting, we can show that 
the topological space $\mSpec \OO$ 
is discrete if and only if 
$\OO$ has the  
finite character property, that is 
every element of $\OO$ is contaioned in 
at most finitely many maximal ideals.
(See Lemma \ref{lemdisc}.)
For instance, if $\OO$ is a Dedekind domain, 
or if $\OO$ is a local ring, then 
$\OO$ has the  
finite character property. 
Note that any $\wR$-valued function on a discrete space 
is upper $\wR$-semicontinuous. 
This implies that 
if $\OO$ has the  
finite character property, 
then we have natural isomorphisms
\begin{equation}\label{eqfinsupp}
\begin{cases}
\ccSb(\OO; \wV) 
 = \ccSr(\OO; \wV) 
\simeq \displaystyle
\bigoplus_{\PP \in \mSpec \OO} \wV_\PP, 
& \\[6mm]
\ccSr(\OO; \val ) 
 \simeq \displaystyle
\bigoplus_{\PP \in \mSpec \OO} \val_\PP 
\simeq \bigoplus_{\PP \in \mSpec \OO} v_\pp(L^\times). 
& 
\end{cases}
\end{equation}
Moreover, we can show that 
the equality 
$\ccSb(\OO; \wV) 
 = \ccSr(\OO; \wV) $
holds if and only if $\OO$ has the  
finite character property.
(See \ref{lemdiscb=ref}.) 
So, we deduce that $\cF (\OO)= \Reg \cF (\OO)$ 
if and only if $\OO$ has the  
finite character property.
\end{rem}

Let us observe two typical examples of 
the description of fractional ideals 
by using Theorem \ref{thmmain1} 
in the case when $\OO$ has the  
finite character property.

\begin{ex}\label{exded}
Suppose that $\OO$ is a Dedekind domain. 
(For example, if $L/K$ is a finite extension, 
then $\OO$ is clearly a Dedekind domain.)
Then, Theorem \ref{thmmain1} 
and the isomorphisms 
(\ref{eqfinsupp}) imply that 
$\wF$ induces the isomorphism
\begin{equation}\label{eqprimefact}
\cF(\OO) 
= \Reg \cF(\OO )
= \cI( \OO )
\xrightarrow{\ \simeq \ }
\ccSr(\OO; \val ) 
\simeq 
\bigoplus_{\PP \in \mSpec \OO}  v_\pp(L^\times) 
\end{equation}
sending each $I \in \cF(\OO)$ to
\[
\left(
\min_{ x \in I} v_\PP (x )
\right)_{\PP}
=\left(
\frac{\text{
the multiplicity of $\PP$ in $I$
}}{
(v_\PP(L^\times): \Z)}
\right)_\PP 
\in \bigoplus_{\PP \in \mSpec \OO}  v_\pp(L^\times).
\] 
Note that for any $\PP \in \mSpec \OO$, 
the index $(v_\PP(L^\times): \Z)$ is finite
since we assume that $\OO$ is a Dedekind domain, here. 
So, the isomorphism (\ref{eqprimefact})
is just the prime ideal factorization. 
\end{ex}

\begin{ex}
Let $\oo:=\C[[T]]$ be the ring of 
formal power series in the indeterminate $T$ 
over $\C$, and $K:=\C((T))$ the field of formal Laurent series. 
We define $L$ to be the field of Puiseux series in $T$ over $\C$, 
namely $L:=\bigcup_{n >0} \C((T^{1/n}))$. 
Then, the integral closure $\OO$ of $\oo$ in $L$ coincides 
with the subring $\bigcup_{n >0} \C [[T^{1/n}]]$ of $L$. 
In this case, the ring $\OO$ becomes a local ring 
with the unique maximal ideal 
$\PP:= \bigcup_{r \in \Q_{>0}} T^r \OO$. 
Clearly, it holds that $v_\PP (L^\times)=\Q$, 
so we have $V_\PP=\R$. Therefore, 
by Theorem \ref{thmmain1}, we obtain
\[
\wF \colon \cF(\OO) \xrightarrow{\  \simeq \ } 
\ccSb(\OO; \wV) \simeq \wV_\PP \simeq \Q \sqcup \R. 
\] 
In this case, 
the inverse of the map $\wF$ is given by 
\[
\R \times \{ 0,1\} \supseteq \wV_\PP \ni
(\alpha, s ) \longmapsto 
\left\{ x \in L \mathrel{\big\vert} 
v_\PP (x) \>'_s  \alpha \right\} 
\in \cF(\OO), 
\]
where the symbol $\displaystyle\>'_{s}$ denotes 
the symbol $\ge$ (resp.\ the symbol $>$) 
if $s=0$ 
(resp.\ $s=1$). 
Note that we have $f_{\PP}=f_{\OO}=0$. 
So, the ideals $\OO$ and $\PP$ cannot be distinguished 
only by the $\R$-valued functions $f_\PP$ and $f_\OO$. 
This is an essential reason why we introduce functions 
valued in $\wR$.
(We also note that $\Max(\PP)=\emptyset$, 
and $\Max(\OO)=\{ \PP \}$. So, the $\wR$-valued functions 
$\wf_{\PP}$ and $\wf_{\OO}$ are distinct.)
By Theorem \ref{thmmain1}, we also deduce that 
$\Reg \cF(\OO) = \cF(\OO)$, and $\cI (\OO) \simeq \val_\PP \simeq \Q$. 
In particular, a fractional ideal $I \in \cF(\OO)$ is 
invertible if and only if there exists a rational number $\alpha \in \Q$ 
such that $I= T^\alpha \OO$. 
\end{ex}

For another example for cases when $\OO$ has 
the finite character property, 
see Remark \ref{remKM}, where we give 
notes on the work of Konomi--Morisawa in \cite{KM}. 
In \ref{secEx}, 
we can also see 
some examples where $\OO$ does not have 
the  
finite character property.

\begin{rem}[Naturality of $\wF$]\label{remnatural}
In this article, we call a subring $\oo'$ of $K$ 
{\it a Dedekind order of $K$}
if $\oo'$ is a Dedekind domain, and 
if the fractional field of $\oo'$ 
coincides with $K$.
Here, we shall note that by varying  
a Dedekind orders $\oo'$ and 
algebraic extension fields $L'$ 
of a fixed field $K$, 
the isomorphisms $\wF_{\oo', L'}$ defines 
a natural isomorphism between certain two functors
from $\Dom (K)$ to  $\AbMon$, where 
$\AbMon$ denotes the category of abelian monoids, and   
$\Dom (K)$ is a category defined as follows: 
\begin{itemize}
\item[{\rm (Obj)}] The objects of $\Dom (K)$ are pairs $(\oo', L'/K)$ 
consisting of a Dedekind order $\oo'$ of $K$, 
and an algebraic extension $L'/K$. 
\item[{\rm (Mor)}] Let $(\oo_1', L_1'/K)$ and 
$(\oo_2', L_2'/K)$ be any two objects of $\Dom (K)$. 
We define the set $\Hom ((\oo'_1,L'_1), (\oo'_2,L'_2))$ 
of morphisms from $(\oo_1', L_1'/K)$ to  
$(\oo_2', L_2'/K)$ in  $\Dom (K)$ by 
\[
 \Hom ((\oo'_1,L'_1), (\oo'_2,L'_2)) = 
\begin{cases}
\Hom_{\text{\rm $K$-alg}} (L'_1, L'_2) & 
(\text{if $\oo'_1 \subseteq \oo'_2$}), \\
\emptyset & (\text{if $\oo'_1 \not\subseteq \oo'_2$}),
\end{cases}
\]
where $\Hom_{\text{\rm $K$-alg}} (L'_1, L'_2)$ 
denotes the set of all embeddings of the field
$L'_1$ into $ L'_2$ over $K$. 
\end{itemize}
For each object $(\oo', L'/K)$ of $\Dom (K)$, we define 
$\OO(\oo';K)$ to be the integral closure of   
$\oo'$ in $L$. 
Then, we define two functors 
$\Dom (K) \longrightarrow \AbMon$ 
denoted by $\cF$ and $\cS$ respectively 
as follows. 
\begin{itemize}
\item[$(\cF)$] For each  object $(\oo', L'/K)$ of $\Dom (K)$, 
we define $\cF(\oo', L'/K)$ to be the abelian monoid 
$\cF(\OO(\oo'; L'))$ of fractional ideals of the ring $\OO(\oo'; L')$. 
For each morphism $\eta \colon (\oo_1', L_1'/K) \longrightarrow 
(\oo_2', L_2'/K)$ in $\Dom (K)$, we define a monoid homomorphism
\[
\cF (\eta) \colon \cF(\OO(\oo_1'; L_1')) \longrightarrow 
\cF(\OO(\oo_2'; L_2'));\ 
I \longmapsto \eta (I) \OO(\oo_1'; L_1').
\] 
\item[$(\cS)$] For each  object $(\oo', L'/K)$ of $\Dom (K)$, 
we define $\cS(\oo', L'/K)$ to be the abelian monoid 
$\ccSb(\OO(\oo'; L') ; \wV)$. 
Let $\eta \colon (\oo_1', L_1'/K) \longrightarrow 
(\oo_2', L_2'/K)$ be a morphism in $\Dom (K)$. 
We put $\OO'_i := \OO(\oo_i'; L_i') $  for each $i \in \{ 1,2\}$. 
Then, we define a monoid homomorphism
\[
\cS(\eta) \colon \ccSb(\OO'_1 ; \wV) \longrightarrow 
\ccSb(\OO'_2 ; \wV);\ \wg \longmapsto \wg \circ \res_\eta.
\] 
Here, the map $\res_\eta \colon 
\mSpec \OO'_2 \longrightarrow \mSpec \OO'_1 $
is defined by $\PP' \longmapsto \eta^{-1} (\PP')$
for each $\PP' \in \mSpec \OO_2'$. 
Note that the map $\res_\eta$ 
is a continuous since for any $A \in \Fin (\OO'_1)$, 
it holds that $\res_\eta^{-1} (\cU_{\OO'_1}(A))= 
\cU_{\OO'_2} (\eta(A))$.  
This implies that for any  $\wg \in \ccSb(\OO'_1 ; \wV)$, 
the function 
$\wg \circ \res_\eta$ belongs to $\ccSb(\OO'_2 ; \wV)$.
\end{itemize}
By the definition of the map $\wF$, 
we deduce that the collection of isomorphisms 
\[
\left\{ \wF_{\oo', L'} \colon 
\cF (\oo',L')=\cF(\OO(\oo'; L')) \longrightarrow 
\ccSb(\OO(\oo'; L') ; \wV) =\cS (\oo',L')
\right\}_{(\oo'; L')
\in \Dom (K)}
\]
forms a natural isomorphism $\cF \xrightarrow{\ \simeq \ } 
\cS$. 
Indeed, the following assertion is easily verified:
\begin{enumerate}[] 
\item {\it Let $\eta \colon (\oo_1', L_1'/K) \longrightarrow 
(\oo_2', L_2'/K)$ be a morphism in $\Dom (K)$, and  
put $\OO'_i := \OO(\oo_i'; L_i') $  for each $i \in \{ 1,2\}$. 
Then, the diagram 
\[
\xymatrix{
\cF(\OO'_1) \ar[rr]^{\wF_{\oo'_1,L'_1}} \ar[d]_{\cF(\eta)} 
& \hspace{5mm} & \ccSb(\OO_1'; \wV)\ar[d]^{\cS(\eta)} \\
\cF(\OO_2') \ar[rr]_{\wF_{\oo_2',L'_2}}  
&& \ccSb(\OO_2'; \wV)
}
\]
commutes. 
}
\end{enumerate}
\end{rem}

\subsection{Ideal class semigroups}\label{ssisg}

Here, we shall describe the monoid structure of 
the ideal class semigroup 
(modulo the Picard group) 
by using Theorem \ref{thmmain1}. 

Let us set the notation related to 
the ideal class semigroups. 
Recall that 
the ideal class semigroup of $\OO$
is defined by 
$\Cl (\OO):= \cF (\OO)/\cP(\OO)$, and 
the Picard group of $\OO$ 
is defined by 
$\Pic (\OO):= \cI (\OO)/\cP(\OO)$. 
We write the image of $I \in \cF(\OO)$ 
in $\Cl (\OO)$ by $[I]$.  
Since $\cP (\OO)$ is a group, 
the fractional ideal $I $ 
belongs to $\Reg \cF (\OO)$
if and only if the class $[I]$ 
belongs to $\Reg \Cl (\OO)$.  
By Theorem \ref{thmmain1}, 
we immediately obtain the following.

\begin{cor}\label{corthmmain}
The map $\wF$ induces a monoid isomorphism 
\[
\Cl (\OO)/\Pic (\OO) \xrightarrow{\ \simeq \ }  
\ccSb(\OO ; \wV)/\ccSr(\OO ; \val);\ 
[I] \longmapsto \wf_I \mathrm{mod}\ \ccSr(\OO ; \val).
\]
Moreover, the image of the monoid 
$\Reg \Cl (\OO) /\Pic (\OO) $ by 
this isomorphism is equal to 
\(
\ccSr(\OO ; \wV)/\ccSr(\OO ; \val)
\).
\end{cor}

For any $\PP \in \mSpec \OO$ satisfying 
$V_\pp = v_{\PP} (L^\times)$, 
we have $\wV_\PP=\val_\PP$ by definition. 
So, we obtain the following corollary.

\begin{cor}\label{corramempty}
If we have $V_\PP = v_{\PP} (L^\times)$ for any 
$\PP \in \mSpec \OO$, then 
it holds that $\Reg \Cl (\OO) = \Pic (\OO)$. 
\end{cor}

Let us describe the assertion of 
Corollary \ref{corthmmain} 
more simply in the cases when 
$\OO$ has the  
finite character property. 
By definition,  we have 
\[
\wV_\PP / \val_\PP = 
\begin{cases}
\{0 \} \sqcup \R/v_\PP (L^\times) & 
(\text{if $V_\PP= \R$}), \\
\{0 \}  & (\text{if $V_\PP \ne \R$}).
\end{cases}
\]
So, by the isomorphisms (\ref{eqfinsupp}) 
in Remark \ref{remfindex}, 
we obtain the following corollary.

\begin{cor}\label{corClifford1}
The monoid $\Cl (\OO)$ is a Clifford semigroup, namely 
it holds that  $\Reg 
\Cl (\OO)=\Cl (\OO)$ if and only if 
the ring $\OO$ has the  
finite character property. 
Moreover, 
if ring $\OO$ has the  
finite character property, we have 
\[
\Cl (\OO) \simeq \bigoplus_{\PP}  
\left( \{0 \} \sqcup \R/v_\PP (L^\times) \right),
\]
where $\PP$ runs through all maximal ideals of $\OO$ 
satisfying $V_\PP = \R$. 
\end{cor}

\begin{rem}[Cliffordness of the ideal class semigroup]\label{remClifford}
Bazzoni have proved 
that for a given Pr\"ufer domain $R$, 
the ideal class semigroup $\Cl (R)$ 
becomes a Clifford semigroup 
if and only if $R$ has the  
finite character property 
(\cite{Ba1} Theorem 2.14). 
Since the ring $\OO$ is a Pr\"ufer domain 
(see Corollary \ref{corPrufer}), 
the assertion of 
Corollary \ref{corClifford1} can be regarded 
as refinement of \cite{Ba1} Theorem 2.14
restricted to 
a special cases when the ring  $R$ ($=\OO$)
is a one dimensional normal domain obtained by 
an integral extension of a Dedekind domain.
We also note that 
our proof of Corollary \ref{corClifford1}
(via the application of Theorem \ref{thmmain1}) is  
another proof 
of  \cite{Ba1} Theorem 2.14 for our settings
via a topological and valuation theoretical approach.  
For examples of the cases when $\Cl (\OO) \ne \Reg \Cl (\OO)$ 
in our situations, see Example \ref{exell} and Example \ref{exab}.
\end{rem}

\subsection{Galois equivariant theory}\label{ssGal}
In this subsection, we assume that $L/K$ is Galois, 
and put $G:=\Gal (L/K)$. 
Then, 
on the one hand, 
we define the left action of $G$ 
on $\cF(\OO)$ by the usual way, namely 
$(\sigma, I) \longmapsto \sigma (I) :=\{ 
\sigma (x) \mathrel{\vert} x \in I \}$ 
for each $\sigma \in G$ and $I \in \cF(\OO)$.

Now, let us fix some notation related to 
Hilbert's ramification theory and valuations.
For each $\pp \in \mSpec \oo$,  
we fix a maximal ideal $\pp_L$ of $\OO$ above $\pp$, 
and denote by $D_{\pp}$ (resp.\ by $I_{\pp}$) 
the decomposition (resp.\ inertia) subgroup of $G$ at $\pp_L$. 
(Since $\OO$ is integral over $\oo$, 
such  maximal ideal $\pp_L$ exists by going up theorem.)
The Krull topology on $G$ induces 
a topology on the quotient set $G/D_\pp$.
We define the left action of $G$ on $G/D_\pp$ 
by the left translation. 
Note that  $G/D_\pp$ is a space parameterizing 
the maximal ideals of $\OO$ lying above $\pp$. 
Indeed, we have the following bijection: 
\begin{equation}\label{eqhomeoG/D}
\coprod_{\pp \in \mSpec \oo} G/D_\pp \longrightarrow 
\mSpec \OO ;\ 
\sigma D_\pp \longmapsto \sigma (\pp_L).
\end{equation}
We define  
the right action of $G$ on 
$G/D_\pp$ 
(resp.\ on $\mSpec \OO$)
by $(\tau D_\pp, \sigma) \mapsto \sigma^{-1} \tau D_\pp$
(resp.\ 
by $(\PP, \sigma) \mapsto \sigma^{-1} (\PP)$).  
Then, clearly, the bijection (\ref{eqhomeoG/D})
preserves the right actions of $G$. 
Later in \S \ref{ssclasstop}, we also see that 
the bijection (\ref{eqhomeoG/D}) is 
a homeomorphism.
(See Corollary \ref{corGhomeo}.)

Let $\pp \in \mSpec \oo$ be any element.
Take the additive valuation 
$v_{\pp,L} \colon L^\times \longrightarrow \Q$
at $\pp_L$ 
normalized by $v_{\pp,L}(K^\times)=\Z$, 
namely put $v_{\pp,L}:=v_{\pp_L}$. 
Note that 
for any $\sigma \in G$, we have 
$\pp_{\sigma(\pp_L)}=v_\pp \circ \sigma^{-1}$. 
We define a set $\Ram (\oo)$ by 
\[
\Ram (\oo):= \{ 
\qq \in \mSpec \oo \mathrel{\vert} v_{\qq,L}(L^\times)
\ne V_\qq
\}.
\]
Namely, the set $\Ram (\oo)$ consists of 
all the elements $\qq \in \mSpec \oo$
such that $v_{\qq,L}(L^\times)$ 
are non-discrete (hence dense) subgroups of $\R$.
Since $L/K$ is Galois, the set $\Ram (\oo)$ 
is independent of the choice of primes 
$\qq_L \in \mSpec \OO$ above $\qq \in \mSpec \oo$. 
If the residue field $\oo/\pp$ is perfect, 
then the condition that $\pp$ belongs to $\Ram (\oo)$
is equivalent to that 
$I_\pp$ is infinite.
We put $\val_{\pp,L}:=\val_{\pp_L}$ 
and $\wV_{\pp,L}:=\wV_{\pp_L}$. 
By the homeomorphism (\ref{eqhomeoG/D}), 
we obtain the natural $G$-equivariant isomorphisms
\begin{align*}
\ccSb (\OO;\wV)& \simeq \bigoplus_{\pp \in \mSpec \oo}
\ccSb( G/D_\pp ; \wV_{\pp,L}), \\
\ccSr (\OO ;\wV)& \simeq  \bigoplus_{\pp \in \mSpec \oo}
 \ccSr(G/D_\pp; \wV_{\pp,L}), \\
\ccSr (\OO ;\val)& \simeq \bigoplus_{\pp \in \mSpec \oo}
 \ccSr(G/D_\pp; \val_{\pp,L}),  
\end{align*}
where if $X$ is a topological space with 
a right action $\rho$ of $G$ such that 
$\rho_\sigma \colon X \longrightarrow X$ 
is continuous 
for each $\sigma \in G$, then 
we define the left action of $G$ on $\ccSb (X; \wR)$ by 
$\sigma \wg:= \wg \circ \rho_{\sigma}$ for each 
$\sigma \in G$ and $\wg \in \ccSb (X; \wR)$. 
The naturality of $\wF$ noted in 
Remark \ref{remnatural} implies that 
the action of $G$ is preserved by  $\wF$. 
So, by Theorem \ref{thmmain1}, 
we obtain the following corollary.

\begin{cor}\label{corthmmain1}
The following hold. 
\begin{enumerate}[{\rm (i)}]
\item The isomorphism $\wF$ induces  
a $G$-equivariant monoid isomorphism
\[
(\wF_\pp)_{\pp \in \mSpec \oo} 
\colon \cF(\OO) \xrightarrow{\ \simeq \ } 
\bigoplus_{\pp \in \mSpec \oo}
\ccSb( G/D_\pp ; \wV_\pp).
\] 
\item The image of $\Reg \cF (\OO)$ by $(\wF_\pp)_\pp$ coincides with 
$\bigoplus_{\pp \in \mSpec \oo}
 \ccSr(G/D_\pp; \wV_\pp)$.
\item The image of $\cI (\OO)$ by $(\wF_\pp)_\pp$ coincides with 
$\bigoplus_{\pp \in \mSpec \oo}
 \ccSr(G/D_\pp; \val_\pp)$. 
\end{enumerate}
\end{cor}

We define quotient monoids 
\begin{align*}
\cM_\pp:&= \ccSb (G/D_\pp; \wV_\pp ) /
\ccSr(G/D_\pp; \val_\pp ), \\
\cN_\pp:&= \ccSr(G/D_\pp; \wV_\pp )/
\ccSr(G/D_\pp; \val_\pp ). 
\end{align*}
The $G$-action on 
$\ccSb(G/D_\pp; \wV_\pp )$ 
induces the $G$-actions on  
$\cM_\pp$ and $\cN_\pp$. 
By Corollary \ref{corthmmain}, 
we immediately obtain the following. 

\begin{cor}\label{corcorthmmain}
We have a natural $G$-equivariant isomorphism
\[
\Cl (\OO)/\Pic (\OO)  \simeq  
\bigoplus_{\pp \in \mSpec \oo} \cM_\pp
\]
of abelian monoids. Moreover, the image of 
$\Reg \Cl (\OO) /\Pic (\OO) $ by 
this isomorphism is equal to 
$\bigoplus_{\pp \in \mSpec \oo} \cN_\pp
=\bigoplus_{\pp \in \Ram (\oo)} \cN_{\pp}$.
\end{cor}

\begin{cor}\label{corClifford2}
If $\OO$ has the  
finite character property, then we have 
\[
\Reg 
\Cl (\OO)/\Pic (\OO) \simeq \bigoplus_{\pp \in \Ram(\oo)}
\,   
\bigoplus_{\sigma D_\pp \in G/D_\pp} 
\bigg(\{0 \} \sqcup \R/v_\pp (L^\times)
\bigg).
\] 
\end{cor}

\begin{rem}\label{remKM}
Let $p$ be a prime number. 
If $L/K$ is the cyclotomic $\Z_p$-extension of $\Q$, 
then the $G$-monoid structure of $\Cl (\OO)$ 
was studied by Konomi and Morisawa in \cite{KM}. 
In this case, they proved that  
\begin{eqnarray}\label{eqisomClcyc}
\Cl (\OO) 
\simeq \Pic (\OO) \oplus \cM_{p \Z}
\simeq 
\Pic (\OO) \sqcup 
\bigg(
\Pic (\OO) \oplus \R/\Z[1/p]
\bigg).
\end{eqnarray}
Our Corollary \ref{corcorthmmain} and 
Corollary \ref{corClifford2} 
can be regarded as a generalization of 
their result modulo $\Pic (\OO)$. 
Note that formal application of 
Corollary \ref{corcorthmmain} and 
Corollary \ref{corClifford2} to this case 
implies only that there exist isomorphisms 
\[
\Cl (\OO) /\Pic (\OO) \simeq \cM_{p \Z} 
\simeq \{ 0 \} \sqcup \R/\Z[1/p].
\] 
The existence of 
the isomorphism (\ref{eqisomClcyc})
follows from the special circumstance in this case 
that any invertible  ideal $I$ of $\OO$ 
satisfying $I \cap \Z=p^n \Z$ for some $n \in \Z$
is principal.  
\end{rem}

\begin{rem}[Underlying monoid structure and topology]\label{remtopsp}
For any submonoid $M$ of $\wR$, 
the monoid structure of $\ccSb (G/D_\pp ; M)$ 
and $\ccSr (G/D_\pp ; M)$ (if defined)
depend only on the homeomorphism class of 
the topological space $G/D_\pp$. 
Note that any compact Hausdorff
space without isolated point which has 
a countable open base consisting of clopen subsets
is homeomorphic to the Cantor set $C$ equipped with 
the subspace topology induced by $\R$. 
So, if $L$ is a composite of countably many 
finite extension fields of $K$, and if 
the index of $D_\pp$ in $G$ is infinite, 
then $G/D_\pp$ is 
homeomorphic to $C$. 
In general, 
if the index of $D_\pp$ in $G$ is infinite, then 
we have a continuous surjection $G/D_\pp \longrightarrow C$.
In particular, if $\wV_\pp \ne \val_\pp$, and 
if $G/D_\pp$ is infinite, then 
there exists a surjective homomorphism of abelian monoids, 
from $\ccSr(G/D_\pp; \wV_\pp )$  (resp.\ $\ccSb(G/D_\pp; \wV_\pp )$)
onto the monoid consisting of all the $\R$-valued continuous 
(resp.\ bounded upper semicontinuous) functions on the Cantor set, 
which contains the restrictions of all the $\R$-valued continuous 
(resp.\ bounded upper semicontinuous) functions on the closed interval $[0,1]$.
(For the existence of such surjections, 
see Corollary \ref{corwff}.) 
By this observation, it seems that  
the monoids $\ccSb(G/D_\pp; \wV_\pp )$ and  
$\ccSr(G/D_\pp; \wV_\pp )$ become quite large 
if $\wV_\pp \ne \val_\pp$.
\end{rem}

\section{Upper $\wR$-semicontinuous functions}\label{secmonoid}

In this section, we introduce the notion of 
upper $\wR$-semicontinuous functions, 
and observe their basic properties. 
First, let us define the monoid $\wR$.

\begin{defn}\label{deftotordabmonoid}
As in \S \ref{ssmainthm}, 
we put $\wR:= \R \times \{ 0,1 \}$.
We write $(\alpha ,s) < (\beta ,t)$ if and only if 
either of the following holds:
\begin{enumerate}[(i)]
\item We have $\alpha < \beta$. 
\item We have $\alpha =\beta $, and $s<t$. 
\end{enumerate}
For two elements $a,b \in \wR$, 
we write $ a \le b$ if $a=b$, or if $a < b$. 
We define the sum of $(\alpha ,s), (\beta ,t) \in \wR$ by 
$(\alpha ,s) + (\beta,t):= 
(\alpha+ \beta, \max\{ s,t \})$.
As in \S \ref{ssmainthm}, 
we define an involution $\iota$ on $\wR$ by 
$\iota (\alpha, s):=(- \alpha,s)$
for each $(\alpha,s) \in \wR$. 
Throughout this article, we let 
$\pi_0 \colon \wR \longrightarrow \R$ and 
$\pi_1 \colon \wR \longrightarrow \{ 0,1 \}$
be the projections. 
For each $a \in \wR$, we define 
$\wR_{<a}:= \{ x \in \wR \mathrel{\vert} 
x < a \}$. 
\end{defn}

By definition, we can easily check the following. 

\begin{lem}\label{lemtptom}
The triple $(\wR ,+ ,\le )$ is a totally ordered abelian monoid. 
Namely, the set $\wR$ is an abelian monoid 
quipped with the monoid law $+$, 
and the relation $\le$ defined above is a total order on $\wR$ satisfying
$a+c \le b+c$ for any $a,b,c \in M$ with $a \le b$. 
\end{lem}

Let $X$ be a topological space, 
and $M$ a submonoid of $\wR$. 

\begin{defn}\label{defbdduppersemicont}
Let $\wf \colon X \longrightarrow M$ be any function. 
\begin{enumerate}[{\rm (i)}]
\item Let $x \in X$. 
We say that $f$
is {\it upper $\wR$-semicontinuous at $x$} 
if and only if for any $\alpha \in \wR$ with $\wf(x) < \alpha$, 
there exists an open neighborhood $W_\alpha$ of $x$ 
in $X$ satisfying   
$\wf(y) < \alpha$ for any $y \in W_\alpha$.  
\item We say that  $\wf$
is {\it upper $\wR$-semicontinuous} 
if and only if  for any $x \in X$, 
the function $\wf$ is  upper $\wR$-semicontinuous at $x$. 
We denote by $\cSu (X;M)$ 
the set of  $M$-valued upper $\wR$-semicontinuous functions on $X$.
\item We say that $\wf$ is {\it compactly supported} 
if and only if   there exists a compact subset $Y$ of $X$ 
such that $\wf(x)=(0,0)$ for any $x \in X \setminus Y$. 
We denote by $\ccSu (X;M)$ 
the set of $M$-valued compactly supported  
upper $\wR$-semicontinuous functions on $X$.
\item We say that $\wf$ is {\it bounded above} (resp.\ {\it below}) 
if and only if   there exists an element 
$a \in \wR$ such that 
$\wf(x) \le a$ 
(resp.\ $a \le \wf(x)$)
for any $x \in X$. 
If $\wf$ is bounded above and below, 
then we say that $\wf$ is {\it bounded}.
We denote by $\ccSb (X;M)$ 
the set of  $M$-valued bounded, 
compactly supported and  
upper $\wR$-semicontinuous functions on $X$.
\item Suppose that $\iota (M)=M$. 
In this article, we say that $f$ is 
{\it a reflexible upper $\wR$-semicontinuous function on $X$}
if and only if both $\wf$ and $\iota\circ \wf$ are contained in 
$\ccSb (X;M)$. 
We denote by $\ccSr (X;M)$ 
the set of  $M$-valued 
compactly supported  
reflexible 
upper $\wR$-semicontinuous functions on $X$.
\end{enumerate}
\end{defn}

\begin{rem}
By definition, we can easily show that a function 
$\wf \colon X \longrightarrow M$ 
is upper semicontinuous  
if and only if for any $a \in \wR$, 
the set $\wf^{-1} (\wR_{< a})$ is 
open in $X$.  
\end{rem}

Let us see some basic properties of 
upper $\wR$-semicontinuous functions.

\begin{lem}\label{lembdd}
Suppose that $X$ is compact. 
Let $\wf \colon X \longrightarrow M$ be any function. 
If $\wf$ is upper $\wR$-semicontinuous,  
then $\wf$ is bounded above. 
Moreover, if $\wf$ is reflexible 
upper $\wR$-semicontinuous,  
then $\wf$ is bounded. 
\end{lem}

\begin{proof}
Suppose that $\wf$ is $\wR$-semicontinuous.  
Then, the collection  
\(
\{ 
\wf^{-1}(\wR_{< a })\}_{a \in \wR}
\) 
forms an  open covering of $X$. 
Since $X$ is compact, 
there exists a finite subset $A$ of $\wR$
such that $X=\bigcup_{a \in A}\wf^{-1}(\wR_{< a })$. 
Since $(\wR, \le )$ is totally ordered, 
the finite subset $A$ of $\wR$ has 
the maximum element. Put $a_0 := \max A$. 
Then, we obtain $X=\wf^{-1}(\wR_{< a_0 })$. 
This implies that $X$ is bounded above.
By applying the above arguments to $\iota \circ \wf$, 
we obtain the second assertion. 
\end{proof}

The following criterion of upper $\wR$-semicontinuity is useful.

\begin{lem}\label{lemusccriterion}
Let $\wf \colon X \longrightarrow M$ be any function. 
Put $f:= \pi_0 \circ \wf \colon X \longrightarrow \R$. 
Let $x \in X$ be an element. 
The function $\wf$ is upper $\wR$-semicontinuous at $x$ 
if and only if the following {\rm (A)} and {\rm (B)} hold.
\begin{enumerate}[{\rm (A)}]
\item If $\pi_1(\wf(x))=1$, 
then the $\R$-valued function $f$ is 
upper semicontinuous at $x$ in the usual sense. 
\item If $\pi_1(\wf(x))=0$, then 
there exists an open neighborhood $W$ of $x$ in $X$ 
satisfying the following two conditions. 
\begin{enumerate}[{\rm (a)}]
\item We have $f(y) \le f(x)$ for any $y \in W$. 
\item For any $y \in W$ satisfying $\wf(y)=\wf(x)$, 
it holds that $\pi_1 (\wf(y))=0$. 
\end{enumerate} 
\end{enumerate}
\end{lem}

\begin{proof}
Suppose that $\wf$ is 
upper $\wR$-semicontinuous function on $X$.
Put  
$\alpha := f(x)$. 
Let $\beta \in \R$  
be any element satisfying $\alpha < \beta$. 
In particular, we have $\beta \ge \alpha$. 
By the $\wR$-semicontinuity of $\wf$, 
there exists an open neighborhood $W_0$ of $x$ 
such that $\wf(y) < (\beta,0)$ for any $y \in W_0$.
This implies that $f(y) < \beta$ for any $y \in W_0$. 
So, the function $\wf$ satisfies (A).  
Assume that $\pi_1 (f(x))=0$. 
Then, the $\wR$-semicontinuity of $\wf$ implies that 
there exists an open neighborhood $W_1$ of $x$ 
such that for any $y \in W_1$, we have 
$\wf(y) < (\alpha,1)$, namely  
$\wf(y) \le (\alpha,0)$. 
This implies that $\wf$ satisfies (B).  

Conversely, suppose that $\wf$ satisfies (A) and (B). 
Take any $b:=(\beta,t) \in \wR$ 
with $\wf(x)<b$. 
First, we assume that $\pi_1 (f(x))=1$. 
Then, we have $\alpha < \beta$. 
By the assumption (A), there exists an 
open neighborhood $W_3$ of $x$ 
such that for any $y \in W_3$, we have $f(y) < \beta$. 
So, we obtain $\wf(y) < (\beta ,0) \le b$
for any $y \in W_3$. 
This implies that $\wf$ is 
upper $\wR$-semicontinuous at $x$.
Next, we assume that $\pi_1 (f(x))=0$.
Then, it follows from 
the assumption (B) that there exists an 
open neighborhood $W_4$ of $x$ 
satisfying (a) and (b).
Then, for any $y \in W_4$, we obtain   
$\wf(y) \le (\alpha ,0) < b$. 
Hence $\wf$ is 
$\wR$-semicontinuous at $x$.
This completes the proof of 
\ref{lemusccriterion}. 
\end{proof}

By applying Lemma \ref{lemusccriterion} to 
$\wf$ and $\iota \circ \wf$, 
we immediately obtain the following corollary.

\begin{cor}\label{correfcriterion}
Let $\wf \colon X \longrightarrow M$ be any function. 
Put $f:= \pi_0 \circ \wf \colon X \longrightarrow \R$. 
Let $x \in X$ be an element. 
The function $\wf$ is reflexible 
upper $\wR$-semicontinuous at $x$ 
if and only if the following {\rm (A)} and {\rm (B)}  hold.
\begin{enumerate}[{\rm (A)}]
\item If $\pi_1(\wf(x))=1$, 
then the $\R$-valued function $f$ is 
continuous at $x$. 
\item If $\pi_1(\wf(x))=0$, then 
there exists an open neighborhood $W$ of $x$ in $X$ 
such that $f(y) = f(x)$ for any $y \in W$. 
\end{enumerate}
\end{cor}

By Corollary \ref{correfcriterion}, 
we deduce the following.

\begin{cor}\label{corlocconst}
Suppose that $\iota (M)=M$. 
If $\pi_1 (M)= \{ 0 \}$,  
then all elements 
of $\ccSr (X;M)$ are locally constant. 
\end{cor}

By Lemma \ref{lemusccriterion} and 
Corollary \ref{correfcriterion}, 
we also obtain the following corollary. 

\begin{cor}\label{corwff}
Suppose that $M$ contains $\pi_0 (M) \times \{1 \}$. 
For any $\R$-valued
continuous (resp.\ upper semicontinuous)
function $f \colon X \longrightarrow 
\pi_0(M)$, there exists an element 
$\wf \in \ccSr (X;M)$ (resp.\ $\wf \in \ccS (X;M)$)
such that $\pi_0 \circ \wf=f$.
\end{cor}

\begin{proof}
Indeed, for given $f$, 
we can take the function 
$\wf \colon X \longrightarrow M$ defined 
by $\wf (x)=(f(x),1)$ for each $x \in X$. 
\end{proof}

\begin{lem}\label{semicontmon}
The set $\cSu (X;M)$ together with 
the sum of functions 
is an abelian monoid.   
\end{lem}

\begin{proof}
Take any $\wf_1, \wf_2 \in \cSu (X;M)$. 
We need to prove that $\wf_1+\wf_2$ 
is upper $\wR$-semicontinuous. 
Let $x \in X$ be any element. 
Put $f_i:= \pi \circ \wf_i$ for each $i \in \{ 0,1 \}$. 
By Lemma \ref{lemusccriterion}, 
the $\R$-valued functions 
$f_1$ and $f_2$ are upper semicontinuous. 
Since the sum of two $\R$-valued 
upper semicontinuous functions 
is also upper semicontinuous, 
the function $f_1 + f_2$ is upper semicontinuous.
So, by Lemma \ref{lemusccriterion}, 
we deduce that $\wf_1 + \wf_2$ is upper
$\wR$-semicontinuous at $x$
if $\pi_1(\wf_1 (x) + \wf_2 (x))=1$, 
namely if $\pi_1(\wf_1 (x))=1$ or $\pi_1( \wf_2 (x))=1$. 
Suppose that $\pi_1 (\wf_1 (x))=\pi_1 (\wf_1 (x))=0$.  
Then, by Lemma \ref{lemusccriterion}, 
for each $i \in \{ 0,1 \}$, 
there exists an open neighborhood $W_i$ of $x$ in $X$ 
satisfying the following (a) and (b). 
\begin{enumerate}[{\rm (a)}]
\item We have $f_i(y) \le f_i(x)$ for any $y \in W_i$. 
\item For any $y \in W_i$ satisfying $f_i(y)=f_i(x)$, 
it holds that $\pi_1 (\wf_i(y))=0$. 
\end{enumerate} 
Put $W:=W_1 \cap W_2$.
Then, by definition, 
the set $W$ is an open neighborhood of $x$, 
and $f_1 (y) + f_2(y) \le f_1 (x) + f_2(x)$
for any $y \in W$.
Suppose that $y \in W$ satisfies 
$f_1(y) + f_2(y)=f_1(x) + f_2(x)$. 
Then, by the property (a) of 
the sets $W_1$ and $W_2$, 
we deduce that 
$f_1(y) = f_1(x)$ and $f_2(y) = f_2(x)$. 
So, for each $i \in \{ 0,1 \}$
the property (b) of 
$W_i$ implies that 
$\pi_1 (\wf_i(y))=0$. 
Hence  
$\pi_1 (\wf_1(y)+\wf_2(y))=0$.
By Lemma \ref{lemusccriterion}, 
the sum $\wf_1 + \wf_2$ is also upper
$\wR$-semicontinuous at $x$  when 
$\pi_1 (\wf_1 (x))=\pi_1 (\wf_2 (x))=0$.
This implies that $\wf_1 + \wf_2 \in\cSu (X;M)$.
So, the set $\in\cSu (X;M)$
becomes 
a semigroup. 
Since $\cSu (X;M)$ contains 
the constant function $(0,0)$, 
which is the identity of the addition, 
the semigroup $\cSu (X;M)$ is a monoid. 
\end{proof}

The following two corollaries are clear.

\begin{cor}
The sets 
$\ccSu (X;M)$ and $\ccSb (X;M)$ 
are submonoids of $\cSu (X;M)$.
\end{cor}

\begin{cor}\label{correfgrp}
Suppose that $\iota (M)=M$. 
Then, the set 
$\ccSr (X;M)$ is a submonoid of  
$\ccSb (X;M)$. 
Moreover, 
if $\pi_1(M)=\{ 0 \}$,  
then $\ccSr (X;M)$ is a group. 
\end{cor}

\if0
\begin{proof}
Suppose that $\iota (M)=M$.
Note that the element $(0,0)$ is 
clearly contained in $\ccSr (X;M)$.
For each  $\wf_1, \wf_2 \in \ccSu (X;M)$, 
it holds that 
$\iota \circ (\wf_1 +\wf_2)
=\iota \circ \wf_1 + \iota \circ \wf_2$. 
It follows from these facts that 
$\ccSr (X;M)$ is 
a submonoid of $\ccSu (X;M)$.

Let us assume the additional 
hypothesis that $\pi_1(M)=\{ 0 \}$. 
Let $\wf \in \ccSr (X;M)$ 
be any element. 
By Lemma \ref{semicontmon}, 
it suffices to show that 
there exists  
the group theoretic inverse of 
$\wf$ in $\ccSr (X;M)$ . 
Note that we have 
$\iota \circ \wf \in \ccSr (X;M)$. 
Since  we have $a+\iota (a)=(0,0)$
for any $a \in \wR$ with $\pi_1(a)=0$,  
the assumption that $\pi_1(M)=\{ 0 \}$
imply that $\iota \circ \wf$ is 
the group theoretic inverse of $\wf$.
\end{proof}

\fi

\section{A topology on the maximal spectrum}\label{ssclasstop}

In this section, we introduce some topological notions 
related to commutative rings.
In \S \ref{sstop}, we define a topology on 
$\mSpec R$ called Krull-like topology 
for each commutative ring $R$, 
and observe some basic properties. 
In \S \ref{ssval}, we give some remarks on 
valuations on $\OO$. 

\subsection{The Krull-like topology}\label{sstop}

Let $R$ be a commutative ring. 
As in \S \ref{ssmainthm}, 
we denote by $\Fin (R)$  the collection of all finite subsets of $R$, 
and  
for each $A \in \Fin (R)$, 
we put $\cU (A)=\cU_R(A):=\{ \pp 
\in \mSpec R \mathrel{\vert} A \subseteq \pp \}$. 
Note that for any  $A,B \in \Fin (R)$, 
we clearly have
$A \cup B \in \Fin (R)$, and 
$\cU(A) \cap \cU(B)= \cU (A \cup B)$. 
So, we can define a topology on $\mSpec \OO$
whose open base is given by   
$\BB(R):=\{ \cU (A) \mathrel{\vert} A \subseteq \Fin (R) \}$.

\begin{defn}\label{defntop}
We call the topology on $\mSpec R$ 
with the open base $\BB(R)$ 
{\it the Krull-like topology of $\mSpec R$}. 
In this article, we always regard 
$\mSpec R$ as a topological space 
equipped with 
the Krull-like topology. 
\end{defn}

\begin{rem}\label{lemnoethdisc}
Let $R$ be a commutative ring. 
By definition, the following are clear. 
\begin{enumerate}[(i)]
\item For each $\PP \in \mSpec R$, 
the collection $\{ \cU(A) \mathrel{\vert} 
A \in \Fin (\PP)  \}$ forms 
a fundamental system of neighborhoods of $\PP$. 
\item Let $R_0$ be a subring of $R$, 
and $I$ an ideal of $R_0$ 
generated by a finite set $A_I$. 
Then, the set $\cU_R(A_I)$ consists of 
all the maximal ideals of $R$ containing $I$. 
In particular, if $I$ is a maximal ideal of $R_0$, 
then    $\cU_R(A_I)=\{\PP \in \mSpec R 
\mathrel{\vert} \PP \cap R_0 =I \}$.
\item If $R$ is a Noetherian ring, then 
$\mSpec R$  is discrete. 
\end{enumerate}
\end{rem}

Let $(\oo,L/K)$ and $\OO$ be as in \S \ref{secintro}.
We denote by $\IF (L/K)$ the set of 
all intermediate field of $L/K$
which is finite over $K$. 
Take any $M \in \IF (L/K)$. 
We define $\oo_M$ to be the integral closure of $\oo$ in $M$, 
namely $\oo_M := \OO \cap M$. 
Since $M/K$ is finite, 
the ring $\oo_M$ is a Dedekind domain. 
Since the ring $\OO$ is integral over $\oo_M$, 
we have a surjection $\mSpec \OO 
\longrightarrow \mSpec \oo_M$ 
defined by $\PP \mapsto \PP \cap M=\PP \cap \oo_M$.

For each $M,M' \in \IF (L/K)$ with $M \subseteq M$, 
we define a surjection
\[
\mathrm{res}_{M',M}\colon 
\mSpec \oo_{M'} \longrightarrow \mSpec \oo_M
\] 
defined by $\pp \longmapsto \pp \cap M$. 
(Note that the map $\mathrm{res}_{M,M'}$ is clearly 
continuous since Remark \ref{lemnoethdisc} (iii) implies that 
$\mSpec \oo_M$ is discrete for any $M \in IF (L/K)$.)
Since the collection
$\{ \mathrm{res}_{M,M'} \mathrel{\vert} 
M,M' \in \IF (L/K), M \subseteq M' \}$
forms a projective system of topological spaces, 
we can define the projective limit 
\[
\Lim:=\varprojlim_{M \in \IF (L/K)} 
\mSpec \oo_M
\] 
equipped with the limit topology. 
For each $M \in \IF (L/K)$, 
we denote by 
\[
\res_M \colon 
\Lim \longrightarrow 
\mSpec \oo_M
\]
the projection to the $M$-component. 
By definition, the map $\res_M$ is continuous. 
Note that the collection 
\[
\BB(\Lim):=
\{ \res_M^{-1} (\pp) \mathrel{\vert} 
M \in \IF (L/K), \ 
\pp \in \mSpec \oo_M  \}
\]
forms an open base of $\Lim$.

\begin{lem}\label{lemprojhomeo}
We have a homeomorphism 
\[
\L \colon 
\mSpec \OO \longrightarrow \Lim ;\ 
\PP \longmapsto (\PP \cap M)_M.
\] 
\end{lem}

\begin{proof}
The bijectivity of the map $\L$ is clear  
since we have the inverse 
\[
\Lim \longrightarrow \mSpec \OO;\ 
(\PP_M)_M \longmapsto 
\bigcup_{M \in \IF(L/K)} \PP_M. 
\]
We need to prove that the map is continuous and open. 
First, let us show that $\L$ is continuous. 
Fix any $W \in \BB(\Lim)$. 
Then, there exist $M \in \IF (L/K)$ and
$\pp \in \mSpec \oo_\M$ 
such that $W:= \res_{M}^{-1} (\pp)$.
Since $\oo_M$ is Noetherian, 
there exists a finite subset $A$ of the ring $\oo_M$ 
generating the ideal $\pp$. 
By definition, we have 
\[
\L^{-1} (W) = \{ 
\PP \in \mSpec \OO \mathrel{\vert} 
\PP \cap M = \pp
\}
=\cU(A).
\]
This implies that 
$\L^{-1}(W)$ is open in $\mSpec \OO$. 
So, the map $\L$ is continuous. 

Next, let us show that $\L$ is an open map.
Fix a finite subset $A$ of $\OO$, and put
$M:=K(A) \in \IF (L/K)$. 
Since $A$ is contained in $\oo_M=\OO \cap M$, 
we have 
\[
\L (\cU (A)) = \res_M^{-1} \left(
\{\pp \in \mSpec \oo_M \mathrel{\vert} 
A \subseteq \pp \}
\right) 
\]
Since $\res_M$ is continuous, 
and since $\mSpec \oo_M$ is discrete, 
we deduce that $\L (\cU (A))$ is open in $\Lim$. 
This implies that $\L$ is an open map. 
\end{proof}

\begin{cor}\label{corcUAcpt}
Let $A \in \Fin (\OO)$ be any element 
containing a non-zero element.  
Then, the subset $\cU(A)$ of $\mSpec \OO$ 
is compact. 
\end{cor}

\begin{proof}
Fix any $A \in \Fin (\OO)$. 
We may assume that $0$ does not belongs to $A$
since we have $\cU(A)=\cU(A \setminus \{0 \})$. 
Take any $M \in \IF (L/K)$ containing $A$.   
Note that the set $\cU_{\oo_M}(A)
= \{ \pp \in \mSpec \oo_M
\mathrel{\vert} \pp  \supseteq A \}$ 
is finite since $\cU_{\oo_M}(A)$ 
coincides with 
the set of all prime ideals of $\oo_M$
dividing the non-zero ideal 
$\oo_M\prod_{a \in A} a $, 
and since $\oo_M$ is a Dedekind domain. 
Similarly to the proof of Lemma \ref{lemprojhomeo}, 
we can easily show that the map
\[
\cU(A) \longrightarrow 
\varprojlim_{A \subseteq 
M \in \IF (L/K)} \cU_{\oo_M}(A);\ 
\PP' \longmapsto (\PP' \cap M)_M. 
\]
is a homeomorphism. 
So, Tychonoff's theorem implies that 
$\cU(A)$ is compact.  
\end{proof}

\begin{cor}\label{corhlc}
The topological space $\mSpec \OO$
is Hausdorff and locally compact.
\end{cor}

\begin{proof}
The topological space $\Lim$ is Hausdorff 
since for any $M \in \IF (L/K)$, 
the topological space is 
$\mSpec \oo_M$ is discrete, 
in particular Hausdorff.  
So, by Lemma \ref{lemprojhomeo}, 
$\mSpec \OO_M$ is Hausdorff.  
Corollary \ref{corcUAcpt} implies that  
$\mSpec \OO$ is locally compact. 
\end{proof}

\begin{cor}\label{corGhomeo}
If $L/K$ is Galois, 
then we have a $G$-equivariant 
homeomorphism 
\[
\coprod_{\pp \in \mSpec \oo} G/D_\pp 
\longrightarrow \mSpec \OO;\ 
\sigma D_\pp \longmapsto \sigma (\pp_{L}),
\]
where we regard the set 
$\coprod_{\pp \in \mSpec \oo} G/D_\pp $ 
as a topological space equipped 
with the coproduct 
topology. 
\end{cor}

\begin{proof}
For each $\pp \in \mSpec \oo$, 
we fix a finite subset $A_\pp$ of the ring $\oo$ 
generating the ideal $\pp$. 
By Remark \ref{lemnoethdisc} (ii), 
we have 
$\cU (A_\pp)
= \{ \PP \in \mSpec \oo_M
\mathrel{\vert} \PP \cap K = \pp \}$
for each $\pp \in \mSpec \oo$.
So, we obtain the coproduct decomposition
\[
\mSpec \OO = 
\coprod_{\pp \in \mSpec \oo} \cU (A_\pp). 
\]
As in \S \ref{ssGal}, 
we fix a prime $\pp_L$ of $\OO$ above $\pp$.
Then, we have homeomorphisms  
\[
G/D_\pp \simeq 
\varprojlim_{M \in \IF (L/K)}
G/D_\pp U_M \xrightarrow{\ \simeq\ } 
\varprojlim_{M \in \IF (L/K)}
\cU_{\oo_M} (A_\pp) 
\simeq \res^{-1}_K(\pp), 
\]
where the second homeomorphism is given by 
$
(\sigma D_\pp U_M)_M
\longmapsto (\sigma (\pp_{L} \cap M))_M$. 
This implies 
the assertion of Corollary \ref{corGhomeo}. 
\end{proof}

Recall that the ring $\OO$ has the  
finite character property
if and only if any non-zero element of $\OO$
is contained in at most finitely many maximal ideals.

\begin{lem}\label{lemdisc}
The ring $\OO$ has the  
finite character property if and only if 
$\mSpec \OO$ is discrete. 
\end{lem}

\begin{proof}
First, suppose that 
$\OO$ has the  
finite character property. 
Fix any $\PP \in \mSpec \OO$, 
and take any non-zero element $x \in \PP$. 
Since $\OO$ has the  
finite character property, 
the open neighborhood $\cU (\{ x \})$
is a finite set. 
This implies that $\mSpec \OO$ is 
discrete since $\mSpec \OO$ is Hausdorff. 
Next, we assume that $\mSpec \OO$ is discrete. 
Take any non-zero element $x \in \OO$. 
Since $\cU (\{ x \})$ is compact 
and discrete, we deduce that
$\cU (\{ x \})$ is  finite. 
Hence $\OO$ has the  
finite character property. 
\end{proof}

\subsection{Valuations on the ring $\OO$}\label{ssval}
Let $\PP \in \mSpec \OO$ be any element. 
Here, 
let us introduce an additive valuation 
$v_\PP \colon L^\times \longrightarrow \Q$
corresponding to $\PP$.
Take any $M \in \IF (L/K)$. 
We put 
$\PP_M:=\PP \cap M$, and 
we denote by 
${\oo}_{\PP,M}$ 
the localization of of $\oo_M$ at $\PP_M$. 
Since ${\oo}_{\PP,M}$ is a DVR, we have a unique additive 
valuation 
$v_{\PP,M} \colon M^\times  \longrightarrow \Q$
normalized by 
\begin{equation}\label{eqnormalization}
v_{\PP,M}(K^\times) = \Z.
\end{equation} 
For convenience, we put $v_{\PP,M}(0)= +\infty$. 
By the normalization (\ref{eqnormalization}), 
the following obviously holds.

\begin{lem}\label{lemgluing}
For any $M_1,M_2 \in \IF (L/K)$
satisfying $M_1 \subseteq M_2$ satisfying 
$M_1 \subseteq M_2$, 
the restriction of $v_{\PP,M_2}$
to $M_1$ coincides with $v_{\PP,M_2}$. 
\end{lem}

\begin{defn}\label{defval}
We define the function 
$v_\PP \colon L ^\times \longrightarrow \Q$
by sending $x \in L^\times $ to 
$v_{\PP,M} (x)$ if $x$ belongs to 
a field $M \in \IF (L/K)$. 
By Lemma \ref{lemgluing}, 
the value $v_\PP (x)$ does not depend on 
the choice of the field $M$ containing $x$.
\end{defn}

\begin{lem}\label{lemdiscb=ref}
The ring $\OO$ has the  
finite character property if and only if 
the abelian monoid $\ccSr(\OO, \wV)$ coincides with 
$\ccSb(\OO, \wV)$. 
\end{lem}

\begin{proof}
Suppose that the ring $\OO$ has the  
finite character property. 
Then, by Lemma \ref{lemdisc},
the topological space 
$\mSpec \OO$ is discrete. So,  
any $\wR$-value function on 
$\mSpec \OO$ is upper $\wR$-semicontinuous. 
This implies that 
$\ccSb(\OO, \wV)=\ccSr(\OO, \wV)$. 

Let us assume that 
$\OO$ does not have the  
finite character property. 
Then, there exists an element 
$\PP \in \mSpec \OO$ 
which is a limit point.
We define a function 
$\wf \colon \mSpec \OO 
\longrightarrow \wR$ by 
$\wf(\PP')=(0,0)$ (resp.\ $\wf(\PP')=(1,0)$)
if $\PP' \ne \PP$ (resp.\ $\PP' = \PP$).  
Then, we have 
$\wf \in \ccSb(\OO, \wV)$, 
and $\wf \notin \ccSr(\OO, \wV)$. 
\end{proof}

\section{Proof of Theorem \ref{thmmain1}}\label{secpf}

In this section, we prove 
Theorem \ref{thmmain1}. 
In \S \ref{sswF}, 
we construct the map $\wF$.
In \S \ref{ssfracideal}, 
we prove the assertion (i)
of Theorem \ref{thmmain1}, 
that is, the map is an isomorphism. 
In \S \ref{ssinvideal},  
we show the assertion (iii). 
In other words, in \S \ref{ssinvideal}, 
we determine the image of 
invertible fractional ideals by the map $\wF$.
In \S \ref{ssreg}, we study 
the image of regular fractional ideals, 
and complete the proof of 
Theorem \ref{thmmain1}.

It is convenient to introduce the following notation. 

\begin{defn}
For each $\PP \in \mSpec \OO$ and each $M \in \IF (L/K)$, 
we define a maximal ideal $\PP_M$ of $\oo_M$ by 
$\PP_M:=\PP \cap M$, and 
fix a finite subset set $A_\PP(M)$ of the ring $\oo_M$ 
which generates $\PP_M$.  
\end{defn}

\subsection{Construction of the map $\wF$}\label{sswF}

Here, let us construct the map 
$\wF$ in Theorem \ref{thmmain1}. 
Recall that in \S \ref{ssmainthm}, 
for a given fractional ideal 
$I \in \cF(\OO)$, 
we have defined 
a function $\wf \colon \mSpec \OO 
\longrightarrow \wR$ by
\[
\wf_{I} (\PP):= \begin{cases}
(f_{I}(\PP), 0) & (\text{if $\PP \in \Max (I)$}), \\
(f_{I}(\PP), 1) & (\text{if $\PP \notin \Max (I)$}),
\end{cases}
\]
where  
$f_{I}\colon \mSpec \OO \longrightarrow \R$ 
is a function given by 
\(
f_{I}(\PP):= \inf
\{ v_{\PP}(x) \mathrel{\vert} x \in I \}
\)
for each $\PP \in \mSpec \OO$, and 
$\Max (I)$ is  a subset of $\mSpec \OO$
consisting of all the primes $\PP$ 
such that the function $v_{\PP}\vert_{I}$ 
has a minimum value. 
In \S \ref{ssmainthm}, we have also defined 
$\ccSb (\OO;\wV)$ to be the set 
consisting of all the bounded, compactly supported and 
upper $\wR$-semicontinuous functions 
$\wg \colon \mSpec \OO \longrightarrow \wR$
satisfying $\wg (\PP) \in \wV_\PP$. 
Note that $\ccSb (\OO;\wV)$ is an abelian monoid. 
Here, let us verify that 
we can define the homomorphism 
$\wF \colon \cF(\OO) \longrightarrow \ccSb( \OO ; \wV)$ 
sending each  
$I \in \cF (\OO)$ to $\wf_I$. 
In order to do this, first, we prove the following lemma.

\begin{lem}\label{lemfI0}
Let $I$ and $J$ be fractional ideals of $\OO$, 
and $\pp$ any maximal ideal of $\oo$. 
Then, the following assertions hold.
\begin{enumerate}[{\rm (i)}]
\item We have $f_{IJ}=f_{I}+f_{J}$.
\item We have $\Max (IJ)=\Max(I) \cap \Max (J)$.
\end{enumerate}
\end{lem}

\begin{proof}
Let us show the assertion (i). 
Fix $\PP \in \mSpec \OO$. 
Take any $\varepsilon \in \R_{>0}$.
Then, there exist elements 
$x \in I$ and $y \in J$
such that $v_\PP (x) \le \wf_I(\PP) + \varepsilon/2$, 
and $v_\PP (y) \le \wf_I(\PP) + \varepsilon/2$. 
So, we have 
\(
f_{IJ}(\PP) \le v_\PP(xy)= v_\PP(x) + v_\PP(y) 
\le \wf_I(\PP) + \wf_J(\PP) + \varepsilon 
\). 
This implies that $f_{IJ} \le f_I + f_J$. 
Take any $z \in IJ$, and  
finitely many elements 
$x_1, \dots, x_n \in I$ and 
$y_1, \dots, y_n \in I$
satisfying $z=\sum_{i=1}^n x_iy_i$. 
Then, we have 
\[
v_\PP (z)  \ge \min_{1 \le i \le n} v_{\PP}(x_iy_i)
\ge \min_{1 \le i \le n} v_{\PP}(x_i)
+\min_{1 \le i \le n} v_{\PP}(x_i) 
\ge f_{I}(\PP) + f_{J}(\PP).
\]
This implies that $f_{IJ}(\PP) \ge f_{I}(\PP) + f_{J}(\PP)$. 
Hence we obtain the assertion (i). 

Let us show the assertion (ii).
First, we take an element $\PP \in 
\Max (I) \cap \Max (J)$.  
Then, there exist non-zero elements $x \in I$
and $y \in J$ such that 
$f_{I}(\PP) = v_{\PP} (x)$
and $f_{J}(\PP) = v_{\PP} (y)$.
By the assertion (i),  
the element $xy \in IJ$ satisfies that
$f_{IJ}(\PP) = v_{\PP} (\PP (xy))$.
So, we have  
$\Max(I) \cap \Max(J) \subseteq 
\Max (IJ)$.
Next, 
we take any 
$\PP \in \mSpec \OO$ satisfying 
$\PP \notin \Max (I)$ or 
$\PP \notin \Max (J)$.
Let $z \in IJ$ be any element. 
Take finitely many elements $x_1, \dots, x_n \in I$ and 
$y_1, \dots, y_n \in J$ satisfying 
$z=\sum_{i=1}^n x_iy_i$. 
Since $\PP$ does not belong to 
$\Max (I) \cup \Max (J)$, we have 
\[
v_\PP ( z) 
\ge \min_{1 \le i \le n}
v_{\PP} (x_i) + 
\min_{1 \le i \le n} v_{\PP} (y_i ) 
 > f_{I}(\PP)+f_{J}(\PP)  
= f_{IJ}(\PP).  
\]
Hence no element $z \in IJ$
satisfies the equality $v_\PP (z) 
= f_{IJ}(\PP)$. 
This implies that $\PP \notin \Max (IJ) $.
So, we obtain 
$\Max (I) \cap \Max (J) \supseteq 
\Max (IJ)$.
\end{proof}

By Lemma \ref{lemfI0}, 
we obtain the following corollary. 

\begin{cor}\label{lemfIp}\label{lemidem}
For any $I,J \in \cF(\OO)$,  
we have $\wf_{IJ}=\wf_{I}+\wf_{J}$.
\end{cor}

Next, we prove the following lemma.

\begin{lem}\label{lemfIsc}
We have $\wf_{I} \in \ccSb (\OO; \wV)$. 
\end{lem}

\begin{proof}
It suffices to show that 
the function $\wf_{I} \colon 
\mSpec \OO \longrightarrow \wR$ 
is bounded, upper $\wR$-semicontinuous, 
and compactly supported.

First, we prove that $\wf$ is compactly supported. 
Take a non-zero element $x \in I$ and 
an element $d \in L^\times$ such that $d I \subseteq \OO$. 
We put $\oo':=\oo_{K(x,d)}=\OO \cap \OO$.
Note that $K(x,y)/K$ is a finite extension, and 
$\oo'$ is a Dedekind domain. 
We define a subset $\Sigma$  of
$\mSpec \oo'$ by
\[
\Sigma := \{ \pp \in \mSpec \oo'
\mathrel{\vert} x \notin {\oo'}^\times_{\pp} \} 
\cup 
\{ \pp \in \mSpec \oo'
\mathrel{\vert} d \notin {\oo'}^\times_{\pp} \}. 
\]
Then, the $\Sigma$ is finite. 
Since $\oo'$ is Noetherian, 
for each $\pp \in \Sigma$, 
we can take a finite subset $A_\pp$ of $\oo'$
generating $\pp$. 
Since the set $\coprod_{\pp \in \Sigma}A_\pp$ is finite, 
Corollary \ref{corcUAcpt} implies that 
the set $U:= \bigcup_{\pp \in \Sigma} \cU(A_\pp)$ 
is compact. 
If $\qq \in \mSpec \OO$ is not contained in 
the compact set $U$, 
then we have 
$f_{I}(\PP)=f_{dI}(\PP)=0$ and $\PP \in \Max (I)
\cap \Max (dI)$ since $\PP \cap \oo'$ 
is prime to $I \cap \oo'$ and $d\oo'$.
This implies that $\wf$ is compactly supported. 

Note that we can also deduce 
that the image of $f_{d\OO}$ is a finite set 
as follows 
For any $\PP \in U$, 
the function $f_{d\OO}$ 
is constant on the open neighborhood $\cU(\PP \cap \oo')$ 
of $\PP$. Indeed, for each $\PP' \in \cU(\oo')$, 
we have $\PP' \cap \oo'=\PP$, and 
$f_{d\OO}(\PP')=v_{\PP'}(d)=v_{\PP}(d)$. 
So, the function $f_{d\OO}$ is locally constant.
Since $f_{d\OO}$ is compactly supported, 
the image of $f_{d\OO}$ is a finite set.

Next, let us prove that $\wf_{I}$ is bounded below.  
Take an element $d \in L^\times$
satisfying $d I \subseteq \OO$.
Then, the function $f_{d I}$ is non-negative, 
in particular, bounded below. 
By the assertion (i), we have 
$f_{d I}= f_{I}+ f_{d \OO}$.  
Note that $f_{d I}= f_{I}+ f_{d \OO}$ is non-negative, 
in particular bounded below.  
So, the function $f_{I}= f_{dI} - f_{d \OO}$ 
is bounded since the image of the function
$f_{d\OO}$ is finite. 
Hence $\wf_{I}$ is also bounded.

Let us show that 
$\wf_{I}$ is upper $\wR$-semicontinuous.
Fix any element $\PP \in \mSpec \OO$, 
and take any $\epsilon \in \R_{>0}$. 
By the definition of $f_{I}$,
there exists an element $x \in I$
satisfying 
$f_{I}(\PP) \le 
v_{\pp}(x) <  f_{I}(\PP) +\varepsilon$. 
Put $M:= K(x)$. 
Put $A:=A_\PP (M)$. 
Then, for any $\PP' \in \cU(A)$, we have 
\begin{equation}\label{eqwfIp}
f_{I}(\PP') \le 
v_{\PP'} (x) = v_{\PP} ( x)
< f_{I}(\PP) +\varepsilon.
\end{equation}
So, the $\R$-valued function $f_{I}$ is upper semicontinuous.
By Lemma \ref{lemusccriterion}, 
we deduce that $\wf_{I}$ is upper $\wR$-semicontinuous
at $\PP$ if $\PP \notin \Max (I)$. 
Suppose that $\PP \in \Max (I)$.
Then, we can take the above $x \in I$ to satisfy that
$f_{I}(\PP) = 
v_{\PP}(x)$. 
Similarly to above, we again put 
$M:=K(x)$, and $A:=A_\PP(M)$. 
Take any $\PP' \in \cU(A)$. 
As we have seen above, 
the inequality (\ref{eqwfIp}) holds.
Moreover, if $f_{I}(\PP') 
=f_I (\PP) $, 
then it holds that $f_{I}(\PP') \in \Max_\pp (I)$
since the inequality (\ref{eqwfIp}) implies that
$f_{I}(\PP') 
 = v_\PP (x)=v_{\PP'} (x)$. 
So, by Lemma \ref{lemusccriterion}, 
we deduce that $\wf_{I}$ is also upper $\wR$-semicontinuous
at $\PP$ if 
$\PP \in \Max (I)$. 
Hence $\wf$ is upper $\wR$-semicontinuous everywhere. 

Since $\wf$ is  upper $\wR$-semicontinuous, 
and since $\wf$ is compactly supported, 
Lemma \ref{lembdd} implies that $\wf$ is bounded above.
This completes the proof of Lemma \ref{lemfIsc}.
\end{proof}

By Corollary \ref{lemfIp} and 
Lemma \ref{lemfIsc}, we obtain the following. 

\begin{prop}\label{propwfdef}
There exists a monoid homomorphism 
\[
\wF \colon \cF(\OO) \longrightarrow 
\ccSb( \OO ; \wV);\ 
I \longmapsto \wF(I):=\wf_I.
\] 
\end{prop}

\begin{rem}
Suppose that $L/K$ is Galois, 
and take any $\PP \in \mSpec \OO$. 
For any $\sigma \in G$ and $\PP \in \mSpec \OO$, 
we have $v_{\sigma^{-1} (\PP)}=v_\PP \circ \sigma$. 
So, by definition, the map $\wF$ 
preserves the action of $G:=\Gal(L/K)$. 
Namely, for any $\sigma \in G$, 
we have $\wf_{\sigma(I)}=\wf_I \circ \rho_{\sigma^{-1}}$, 
where $\rho_{\sigma^{-1}} \colon \mSpec \OO \longrightarrow 
\mSpec \OO$ is a map defined by 
$\rho_{\sigma^{-1}} (\PP)= \sigma^{-1}(\PP)$. 
\end{rem}

\subsection{Parameterization of fractional ideals}\label{ssfracideal}
Here, let us show that
the map $\wF$ constructed in the previous subsection
is an isomorphism, in particular 
that $\wF$ is injective and surjective.
In order to show this, let us set some notation, 
and prove some preliminary results.

Recall that we have defined $\wR:=\R \times \{ 0,1 \}$. 
We denote by 
$\pi_0 \colon \wR \longrightarrow \R$ and 
$\pi_1 \colon \wR \longrightarrow \{ 0,1 \}$
the projections.

\begin{defn}
Let $\wg \in \ccSb( \OO ; \wV)$ 
be any element. 
For each intermediate field $M$ of $L/K$, 
we define 
\[
I_M(\wg):=
\bigcap_{\PP \in \mSpec \OO}
\{ x \in M \mathrel{\vert} v_{\PP} (x) 
\>'_{\pi_1(\wg(\PP))}  
\pi_0 (\wg (\PP)) \}, 
\]
where the symbol $\displaystyle\>'_{s}$ denotes 
the symbol $\ge$ (resp.\ the symbol $>$) 
if $s=0$ 
(resp.\ $s=1$). 
\end{defn}

\begin{lem}\label{lemdedekind}
Let $\wg \in \ccSb( \OO ; \wV)$ 
be any element. 
We define a function $g:=\pi_0 \circ \wg$, 
and a subset $\CC :=(\pi_1 \circ \wg)^{-1}(1) 
\subseteq \mSpec \OO$. 
Take any $M \in \IF (L/K)$. 
Then, the following hold. 
\begin{enumerate}[{\rm (i)}]
\item The set $I_M(\wg)$ is 
a fractional ideal of $\oo_M$. 
\item Take any $\PP \in \mSpec \OO$, and put 
$A_\PP:=A_\PP(M)$. 
We denote by $\oo_{M,\PP}$ the localization of $\oo_{M}$ at $\PP_M$.
We define $e=e_M(\PP_M):=(v_\PP (M^\times): \Z)$.  
A rational number $\alpha=
\alpha_M(\PP_M) \in \Q$ is defined by
\[
\alpha := 
\begin{cases}
\displaystyle\max_{\PP' \in \cU (A_\PP)}
\left(
\min \big( e^{-1}\Z
\cap \R_{ \ge 
g(\PP')} 
 \big) \right)
 & (\text{if $\cU(A_\PP) \cap \CC = \emptyset$}), \\
\displaystyle\max_{\PP' \in \cU(A_\PP)}
\left(
\min \big( e^{-1}\Z
\cap \R_{ > 
g(\PP')}
 \big) \right) & 
(\text{if $\cU(A_\PP) \cap \CC \ne \emptyset$}).
\end{cases}
\]
Then, it holds that  $I_M(\wg)\oo_{M,\PP}
=(\PP_M)^{\alpha e}\oo_{M,\PP}$, 
\end{enumerate}
\end{lem}

\begin{proof}
Fix any $M \in \IF (L/K)$.
Note that $e^{-1}\Z$ is discrete, 
and $g$ is bounded. 
So, the rational  
$\alpha_M (\PP_M)$ is defined 
after we choose $A_\PP$. 
Since $\cU (A_\PP)$ coincides with 
the set consisting of 
the primes containing $\PP_M$, 
the rational  $\alpha_M (\PP_M )$
depends only on the ideal $\PP_M$ of $\oo_M$. 
In other words, we have a function 
$\mSpec \oo_M \longrightarrow \Q$
defined by $\pp \longmapsto  \alpha_M (\pp)$.

Since the function $\wg$ is compactly supported, 
we can take a finite set $B \in \Fin (\OO)$
such that the values of $\wg$ are constantly equal to $(0,0)$
outside $\cU(B)$. We define 
$B_M:=\{ N_{M(x)/M}(x) \mathrel{\vert} x \in B \}$. 
By definition, we have $\cU(B) \subseteq \cU(B_M)$. 
We define a subset $\Sigma_M$ of $\mSpec \oo_M$ by 
\(
\Sigma_M:= \{ \PP_M \mathrel{\vert} 
\PP \in \cU (B_M) \} 
\). 
Note that $\Sigma_M$ is a finite set.
So, we can define a fractional ideal $J$ 
of $\oo_M$ by 
\(
J:= \prod_{\pp \in \Sigma_M} 
\pp^{\alpha_M (\pp)e_M(\pp)} 
\). 
By definition, we have $I=J$.
So, the assertions (i) and (ii)
of Lemma \ref{lemdedekind} hold.
\end{proof}

\begin{lem}\label{lemifideal}
Let $I$ be any fractional ideal $I$ of $\OO$, 
and $\pp$ a maximal ideal of $\oo$. 
For any $M \in \IF (L/K)$, 
we have 
\(
I \cap M = I_M(\wf_I)
\). 
\end{lem}

\begin{proof}
We put $J:=I_M(\wf_I)$. 
By definition, we have $I \cap M \subseteq J$. 
So, it suffices to show that $I \cap M \supseteq J$.
Take any $\PP \in \mSpec \OO$, and
put $A_\PP:=A_\PP (M) $.  
We denote by $\oo_{M,\PP}:=\oo_{M,\PP_M}$ 
the localization of $\oo_M$ at $\PP_M$. 
In order to prove Lemma \ref{lemifideal}, 
it suffices to show that 
\(
(I \cap M) \oo_{M,\PP} \supseteq J \oo_{M,\PP} 
\). 
We set $\alpha:=\alpha_M(\PP_M)$ as in Lemma \ref{lemdedekind}.
By the definition of $\alpha$, 
for any $\PP' \in \cU (A_\PP)$, 
we can take a non-zero element 
$x_{\PP'} \in I$ satisfying 
$v_{\PP'}(x_{\PP'}) \le \alpha$. 
For each $\PP' \in \cU (A_\PP)$, 
we put $A'_{\PP'}:=A_{\PP'}( M(x_{\PP'}))$. 
Then, for any 
$\PP' \in \cU (A_\PP)$, the set 
$\cU(A'_{\PP'})$ becomes an open neighborhood of $\PP'$
contained in $\cU(A_{\PP})$. 
So, we obtain  an open covering
$\{ \cU(A'_{\PP'}) \mathrel{\vert} 
\PP' \in \cU(A_{\PP}) \}$ of $\cU(A_\PP)$. 
Since $\cU(A_\PP)$ is compact,  
there exists a finite subset 
$T \subseteq \cU(A_{\PP})$ such that 
\[
\cU(A_{\PP}) = \bigcup_{\PP' \in T} 
\cU(A'_{\PP'}). 
\]
Let $M' \in \IF(L/K)$ be the field 
generated by the finite set $\{x_{\PP'} 
\mathrel{\vert} \PP' \in T \}$
over $M$. 
Let $\Sigma_{M'}$ be the set of 
all primes of $\oo_{M'}$ lying above $\PP_M$.
For each $\pp \in \Sigma_{M'}$, we 
put $e'_\pp:=(v_{\PP'}({M'}^\times):v_{\PP} (M^\times))$, 
and $e:=(v_{\PP}(M^\times):\Z)$, 
where $\PP'$ is a prime of $\OO$ above $\pp$. 
Note that $e'_{\pp}$
is independent of the choice of $\PP'$ above $\pp$. 
For any $\pp \in \Sigma_{M'}$, 
we have 
\begin{equation}\label{eqIcapM'}
(I \cap M') \oo_{M', \pp} \supseteq \sum_{\PP' \in T} 
x_{\PP'} \oo_{M', \pp} 
\supseteq {\pp}^{\alpha ee'}\oo_{M', \pp}.
\end{equation}
We define a semilocal ring 
$\oo_{M',M,\PP}$ by 
\(
\oo_{M',M,\PP}=\oo_{M'}\otimes_{\oo_M} \oo_{M,\PP}
=\bigcap_{\pp \in \Sigma_{M'}} \oo_{M' , \pp}
\). 
Note that for any fractional ideal 
$\mathfrak{A}$ of $\oo_{M'}$, 
we have 
$\mathfrak{A} \oo_{M',M, \PP} \cap M 
= (\mathfrak{A} \cap M) \oo_{M,\PP}$.
So, by (\ref{eqIcapM'}), and by the prime factorization 
$\PP_M \oo_{M'}= \prod_{\pp \in  \Sigma_{M'}} \pp^{e'_\pp}
=\bigcap_{\pp \in  \Sigma_{M'}}  \pp^{e'_\pp} $, 
it holds that
\begin{align*}
(I \cap M) \oo_{M, \PP}& =
((I\cap M')\cap M)\oo_{M, \PP}=
(I \cap M') \oo_{M',M,\PP} \cap M \\
&\supseteq \bigcap_{\pp \mid \PP}
\pp^{\alpha e e'_\pp}\oo_{M',M, \PP} \cap M
=\PP_M^{\alpha e} \oo_{M',M, \PP} \cap M
=\PP_M^{\alpha e} \oo_{M, \PP} =J \oo_{M, \PP}. 
\end{align*}
Here, the last equality follows from 
Lemma \ref{lemdedekind} (ii).
Hence we obtain $I \cap M \supseteq J$. 
\end{proof}

\begin{cor}\label{corinjfidem}
Let $I$ and $J$ be fractional ideals of $\OO$.
We have $I = J$ if and only if 
we have $\wf_{I}=\wf_{J}$, namely 
$f_{I}= f_{J}$ and 
$\Max(I)= \Max (J)$.
\end{cor}

\begin{proof}
Suppose that 
$\wf_{I}=\wf_{J}$.
Then, by 
Lemma \ref{lemifideal}, 
we have $I \cap M = J \cap M$
for any $M \in \IF (L/K)$.
So,  we obtain $I=J$. 
\end{proof}

\if0
\begin{cor}\label{coridem}
Let $I \in \cF(\OO)$. 
We have $I^2=I$ if and only if $f_{I}=0$.
\end{cor}

\begin{proof}
By Lemma \ref{lemfI0}, 
we have  $f_{I^2}=2f_{I}$ and 
$\Max (I^2)=\Max (I)$. 
So,  
the assertion of Corollary \ref{coridem}
follows from Corollary \ref{corinjfidem}.
\end{proof}

\fi

Now, let us show the assertion (i)
of Theorem \ref{thmmain1}, namely 
that the map 
$\wF$ in Proposition \ref{propwfdef}
is an isomorphism.

\begin{proof}[{\bf Proof of Theorem \ref{thmmain1} (i)}]
It suffices to show that $\wF$ is a bijection. 
The injectivity of $\wF$ follows from 
Corollary \ref{corinjfidem}. 
Let us show that $\wF$ is surjective. 
Take any $\wg \in \ccSb(\OO  ; \wV)$. 
Put $g:= \pi_0 \circ \wg$, 
and $\Xi:=(\pi_1 \circ \wg)^{-1} (0)$.  
We define $I:=I_L(\wg)$.  
Then, let us show that $\wf_I=\wg$, 
namely that $f_{I}=g$, and 
$\Max (I)=\Xi$.

First, let us show that $f_{I}=g$. 
By definition, it holds that  
$f_{I} \ge g$. 
So, it suffices to show that $f_{I,\pp} \le g_\pp$. 
Take any $\PP \in \mSpec \OO$. 
Fix any $\varepsilon \in \R_{>0}$. 
Since $g$ is upper semicontinuous at $\PP$, 
there exists a finite set 
$A_0 \in \Fin (\OO)$ contained in $\PP$ 
such that 
\begin{equation}\label{inequsc}
g(\PP') < g(\PP)+\varepsilon
\end{equation}
for any $\PP' \in U$. 
We take a field $M \in \IF (L/K)$ 
containing  $K(A_0)$ which 
satisfies the following condition $(*)$.
\begin{itemize}
\item[$(*)$] {\it 
If the fixed prime $\PP$ satisfies that $V_\PP=\R$, 
then it holds that $v_\PP (M^\times) \cap 
(0 , \varepsilon) \ne \emptyset $. 
Otherwise, it holds that $g (\PP) \in v_\PP (M^\times)$. }
\end{itemize}
Put $A_1:=A_\PP(M)$. 
Then, by Remark \ref{lemnoethdisc} (ii), 
we have $\cU(A_1) \subseteq \cU (A_0)$. 
So, by the inequality (\ref{inequsc}) 
and Lemma \ref{lemdedekind} (ii), 
we deduce that there exists an element $x \in I \cap M$
satisfying $v_\PP (x) < g (\PP)+2\varepsilon$.
Hence we obtain $f_{I} (\PP ) \le g (\PP) $.

Let us show the assertion (b).
By definition, it holds that 
$\Max (I) \subseteq \Xi$.
So, it suffices to show that 
$\Max (I) \supseteq \Xi$.
Take any $\PP \in \Xi$. 
Put $\alpha:= g (\PP)$. 
Since $\Xi=(\pi_1 \circ \wg)^{-1} (0)$, 
we have $\alpha \in v_\PP (L^\times)$. 
By the upper $\wR$-semicontinuity of $\wg$, 
there exists an a subset $A_2 \in \Fin (\OO)$ 
contained in $\PP$ 
such that $\wg(\PP') < (\alpha, 1)$, 
namely $\wg (\PP') \le g (\PP')=(\alpha, 0)$, 
for any $\PP' \in \cU(A_2)$. 
Let $M' \in \IF (L/K)$ be a finite extension of 
$K(A_2)$ satisfying $\alpha \in v_\PP ({M'}^\times)$. 
We put $A_3:=A_\PP (M')$. 
Then, similarly to above, 
we deduce that $\cU (A_3) \subseteq \cU (A_2)$. 
So, by Lemma \ref{lemdedekind} (ii), 
there exists an element $x' \in I \cap M'$
satisfying $v_\PP (x') = \alpha= g (\PP)$.
Hence we obtain $\PP \in \Max (I)$. 
This implies that 
$\Max (I) = \Xi$.
Now we have proved the assertions (a) and (b).  
So, we obtain $\wf_I= \wg$. 
Therefore, the map $\wF$ is surjective. 
\end{proof}

\subsection{Parameterization of invertible fractional ideals}\label{ssinvideal}

Recall that in \S \ref{ssmainthm}, we have also defined 
$\ccSb (\OO;\val)$ to be the set 
consisting of all the bounded, compactly supported and 
reflexible upper $\wR$-semicontinuous functions 
$\wg \colon \mSpec \OO \longrightarrow \wR$
satisfying $\wg (\PP) \in \val_\PP$. 
Here, let us show the assertion (iii) of 
Theorem \ref{thmmain1}. 
Namely, we shall prove that 
a fractional ideal $I$ of $\OO$ 
is invertible if and only if 
$\wf_{I}$ belongs to $\ccSr(\OO; \val)$. 

\begin{proof}[{\bf Proof of Theorem \ref{thmmain1} (iii)}]
First, suppose that an invertible fractional ideal 
$I \in \cI (\OO)$ is given. 
Let $J$ be the inverse of $I$. 
Take any $\PP \in \mSpec \OO$. 
Lemma \ref{lemfI0} (ii) implies  that 
$\Max (I) \cap \Max (J) 
=\Max (\OO) =\mSpec \OO$. 
So, we can take $x \in I$ and $y \in J$
satisfying $v_\PP ( x)= f_{I}(\PP)$
and $v_\PP ( y)=f_{J}(\PP)$. 
Let $A$ be a finite subset of $\OO \cap K(x,y)$ 
which generates the ideal $\PP \cap K(x,y)$. 
Let $\PP' \in \cU(A)$ be any element. 
Then, we have 
\[
\PP' \cap K(x,y) = \PP \cap K(x,y) \supseteq 
\{ x , y \}
\]
So, it holds that 
$f_{I}(\PP') \le 
v_{\PP'} (x)=v_\PP (x)= 
f_I(\PP)$. Similarly, we have
$f_{J}(\PP') \le f_{J}(\PP)$.
By combining with Lemma \ref{lemfI0} (i), we obtain
\[
0  = f_{\OO}(\PP') 
=f_{I}(\PP')+
f_{J}(\PP') \le f_{I}(\PP)+
f_{J}(\PP)
=f_{\OO}(\PP)
=0.
\] 
Therefore we have
$f_{I}(\PP')=f_{I}(\PP)$
and $f_{J}(\PP')=f_{J}(\PP)$.
This implies that the function 
$f_{I}$ is locally constant.
Hence $\wf$ is contained 
in $\ccSr(\OO; \val)$. 

Next, suppose that 
a fractional ideal $I$ of $\OO$
satisfying $\wf_{I} \in \ccSr(\OO; \val)$ 
is given. 
Let us show that $I$ is invertible. 
Since the function $\wf$ is compactly supported, 
we can take a finite set $B \in \Fin (\OO)$
such that the values of $\wf$ are constantly equal to $(0,0)$
outside $\cU(B)$. 
Since $f_I$ is locally constant, 
and since $\Max (I) =\mSpec \OO$, 
for each $\PP \in \cU(B)$,  
there exist elements 
$B_{\PP} \in \Fin (\PP)$ and $x_{\PP} \in I$
such that $f_I(\PP')=v_{\PP} (x_{\PP})$ 
for any $\PP' \in \cU(B_{\PP})$. 
We may replace $B_{\PP}$ by $B_{\PP} \cup \{ x_{\PP'} \}$ 
(if necessary), 
and suppose that $x_{\PP} \in B_{\PP}$.
Since the collection 
$\{ \cU(B_{\PP}) \mathrel{\vert} \PP \in \cU (B) \}$
forms an open covering of a compact set $\cU (B)$, 
We can take a finite subset $T \subseteq \cU(B)$ 
such that  $\{ \cU(B_{\PP}) \mathrel{\vert} \PP \in T \}$
covers $\cU (B)$. 
Let $M \in \IF (L/K)$ be 
the field generated by $\bigcup_{\PP \in T} B_{\PP}$ 
over $K(B)$. 
We define 
$J_M:= \sum_{\PP \in T} 
\oo_M  x_\PP$.
Lemma \ref{lemdedekind} and 
Corollary \ref{corinjfidem} implies that 
\begin{equation}\label{eqfingen}
I=J_M\OO.
\end{equation}
Since $J_M^{-1}\OO$ is the inverse of $I=J_M\OO$, 
the ideal $I$ is invertible. 
\end{proof}

In the above arguments, 
the equality (\ref{eqfingen})
implies that 
if  $\wf_I$ belongs to $\ccSr(\OO; \val)$, 
then the ideal $I$ is finitely generated. 
Conversely, if a fractional ideal $I$ of $\OO$ 
is generated by a finite subset $A$, 
then the ideal $(I \cap K(A))^{-1} \OO$ 
is the inverse of $I$, where 
$(I \cap K(A))^{-1}$ denotes 
the fractional ideal of the Dedekind domain $\oo_{K(A)}$ 
which is the inverse of $I \cap K(A)$. 
Therefore, we obtain the following corollaries. 

\begin{cor}\label{corPrufer}
A fractional ideal $I$ of $\OO$ is 
invertible if and only if 
$I$ in finitely generated.
In particular, the ring $\OO$ is 
a Pr\"ufer domain.
\end{cor}

\begin{cor}\label{corPiclim}
We have an isomorphism 
$\varinjlim_{M \in \IF (L/K)} \Pic(\oo_M) \simeq 
\Pic (\OO)$.
\end{cor}

\subsection{Parameterization of regular fractional ideals}\label{ssreg}
Here, let us complete the proof of 
Theorem \ref{thmmain1}. 
Since the assertions (i) and (iii) have been 
already proved, it 
suffices to show the assertion (ii), 
which claims that a fractional ideal $I$ of $\OO$ 
is regular if and only if 
$\wf_{I}$ belongs to $\ccSr(\OO; \val)$. 
Recall that  $\ccSb (\OO;\val)$ denotes the set 
consisting of all the bounded, compactly supported and 
reflexible upper $\wR$-semicontinuous functions 
$\wg \colon \mSpec \OO \longrightarrow \wR$
satisfying $\wg (\PP) \in \wV_\PP$. 

\begin{proof}[{\bf Proof of Theorem \ref{thmmain1}}]
Let  $I \in \cF(\OO)$.
Via the isomorphism $\wF$, 
the ideal $I$ belongs to $\Reg \cF (\OO)$
if and only if 
$\wf_{I}$ is regular in 
$\ccSb(\OO ; \wV)$. 
Fix $\PP \in \mSpec \OO$, 
and suppose that $\wf_{I}$ 
is regular in $\ccSb(\OO ; \wV)$. 
Then, there exists a function 
$\wg \in \ccSb(\OO ; \wV)$   
satisfying 
\begin{equation}\label{eqfgh}
\wf_{I}=2\wf_{I}+\wg.
\end{equation}
We put $g:= \pi_1 \circ \wg$.
So, we obtain $f_{I}= -g$.
Since the functions $f_{I}$ and $g$
are upper semicontinuous, 
we deduce that $f_{I}$ is also continuous.

By Lemma \ref{correfcriterion}, 
in order to prove that $\wf_{I} \in \ccSr (\OO; \wV)$, 
it is sufficient to show that for each 
$\PP \in \Max (I)$, there exists
an open neighborhood $W_\PP$ of $\PP$
such that $f_I \vert_{W_{\PP}}$ is constant. 
Fix any $\PP \in \Max (I)$, 
and put $\alpha := f_I (\PP)$. 
By (\ref{eqfgh}), we have $\Max (I) 
\subseteq (\pi_1 \circ g)^{-1}(0)$. 
So, by the upper $\wR$-semicontinuity of 
$\wf_{I}$ and $\wg$, 
there exists 
an open neighborhood $W_{\PP}$ of $\PP$
such that for any $\PP' \in W_{\PP}$, 
all of the following hold.
\begin{enumerate}[(a)]
\item We have $\wf_{I}(\PP') < (\alpha,1)$, namely 
$\wf_{I}(\PP') \le (\alpha,0)$.
\item We have $\wg(\PP') < (g(\PP),1)$, 
namely $\wg(\PP') \le (g(\PP),0)$. 
\end{enumerate}
Since we have $f_{I}= -g$, 
we deduce by (a) and (b) that 
for any $\PP' \in W_{\PP}$, we have 
$\wf_{I}(\PP' )=\alpha$. 
So, we obtain  $\wf_{I} \in \ccSr (\OO; \wV)$.

Conversely, if $\wf_{I,\pp} \in \ccSr (G/D_\pp; \wV)$,
then $\wf_{I}$ 
is regular in  
$\ccSb (\OO; \wV)$ since 
$\iota \circ \wf_{I}$ belongs to 
$\ccSb (\OO; \wV)$, and since we have 
\(
\wf_{I}=
\wf_{I}+ 
\wf_{I}+  
(\iota \circ \wf_{I})
\).
This completes 
the proof of Theorem \ref{thmmain1}. 
\end{proof}

\section{Examples}\label{secEx}

In this section,  
we shall introduce the ideal class semigroup 
of the ring $\OO$ for some $(\oo,L/K)$ 
described by using Theorem \ref{thmmain1}.

\begin{ex}[Tate module covering of elliptic an curve]\label{exell}
Let $E$ be an elliptic curve over $\C$, and 
put $\Lambda:= H_1 (E(\C), \Z)$. 
We fix an isomorphism 
$E(\C) \simeq \C/\Lambda$ of complex Lie groups, 
and identify $E(\C)$ 
with $\C/\Lambda$ via the fixed isomorphism.
Fix a prime  $p$. 
Take any $m \in \Z_{\ge 0}$. 
Let $(p^m \times)_E \colon E \longrightarrow E$
be the multiplication by $p^m$ isogeny, and 
put $E[p^m]:=\Ker (p^m \times)_E$. 
Here, we denote by $E_m$ the scheme $E$ over $E$
whose structure map is given by $(p^m \times)_E$.
We put $E[p^\infty](\C):=\bigcup_{n>0} 
E[p^n](\C)$.  
We define $K:=\C (E_0)$, 
and $L_m:= \C (E_m)$. 
Then, the extension $L_m/K$ is Galois, and we have
$\Gal (L_m/K) \simeq E[p^m](\C) \simeq p^{-m} \Lambda / \Lambda 
\simeq (\Z/p^m\Z)^2$. 
We define $L:=\varinjlim_{m \ge 0} L_m$. 
By definition, we have $G=\Gal (L/K) 
\simeq \varprojlim_m E[p^m](\C)=:T_p(E) \simeq \Z_p^2$. 

For each $n \in \Z_{\ge 0}$, 
we put $R_n:=\Gamma (E \setminus E[p^n], \O_E ) \subseteq K$,  
and  define a subring $\oo$ of $K$ by 
\(
\oo:= \bigcup_{n>0} R_n 
\).  
Then, the ring $\oo$ is a Dedekind domain, and 
we have bijections 
\[
\mSpec (\oo) \simeq E(\C) \setminus   
E[p^\infty](\C) 
\simeq (\C/\Lambda) \setminus (\Z[1/p]\Lambda/\Lambda). 
\]
We identify $\mSpec \oo$ with 
$U:=(\C/\Lambda) \setminus (\Z[1/p]\Lambda/\Lambda)$. 
The natural group isomorphism 
$\Pic^0 (E) \simeq E(\C)$ induces  
\(
\Pic (\oo) \simeq E(\C)/  
E[p^\infty](\C) \simeq \C /\Z[1/p]\Lambda
\). 
We  denote by $\OO$ the integral closure of 
$\oo$ in $L$. 
Then, we have $\OO \cap L_m \simeq \oo$.  
Note that the map 
\[
(p^m \times)_E^* \colon
\Pic (\oo) \longrightarrow 
\Pic (\OO \cap L_m);\ 
[I] \longmapsto [I(\OO \cap L_m))]
\]
is an isomorphism since the map $(p^m \times)_E^*$
composed with the canonical isomorphism 
$\Pic (\OO \cap L_m) \simeq \Pic (\oo)$ 
coincides with the automorphism on $\Pic (\oo)$ defined 
by the multiplication by $p^m$.
In particular, by Corollary \ref{corPiclim}, 
we obtain 
\[
\Pic (\OO) \simeq \Pic (\oo)
\simeq \C/\Z[1/p]\Lambda. 
\]
 
Let $\pp \in \mSpec (\oo)$ be any element. 
We fix a complex number $z_\pp \in \C \setminus 
\Z[1/p]\Lambda$
such that $z_\pp + \Lambda \in U$ coresponds to $\pp$. 
Then, the system 
\[
(p^{-m}z_\pp + \Lambda)_m \in 
\varprojlim_m U = \varprojlim_m \mSpec (\OO \cap L_m)
\]
is naturally identified with a prime ideal $\pp_L$ of $\OO$ 
above $\pp$, where the projective limit 
$\varprojlim_m U$ 
is take with respect to map defined by the multiplication by $p$. 
Note that for any $m \in \Z_{\ge m}$, 
the prime $\pp$ splits completely in 
$\OO \cap K_m$. 
So, we have $D_{\pp}=\{ 1 \}$. 
Hence by Corollary \ref{corramempty}, we obtain 
\(
\Reg \Cl (\OO) = \Pic (\OO) \simeq 
\C/\Z[1/p]\Lambda.
\)
Moreover, since $D_{\pp}=\{ 1 \}$, 
and since $G \simeq T_p(E)$, we have 
\begin{equation*}\label{eqisomell}
\ccSb (G/D_\pp;\wV_\pp)  \simeq 
\cSub (T_p(E) ;\Z)\ \text{and}\ 
\ccSr (G/D_\pp;\wV_\pp) 
 \simeq \Loc (T_p(E) ;\Z), 
\end{equation*}
where $\cSub (T_p(E) ;\Z)$ (resp.\ 
$\Loc (T_p(E) ;\Z)$) denotes 
the abelian monoid (resp. group) of $\Z$-valued bounded 
upper semicontinuous (resp.\ locally constant) 
functions on $T_p(E)$. 
In particular, the abelian $G$-monoid $\Cl(\OO)$
is isomorphic to $\bigoplus_{\PP \in U} \cM_\pp$, 
where the summands
$\cM_\pp$ are 
independent of $\pp \in U=\mSpec \OO$.

In this situation, by using the natural isomorphism 
$\Pic^0 (E) \simeq E(\C)$, 
we can describe the image of $\cP(\OO)$ 
in $\bigoplus_{\pp} \Loc (T_p(E); \Z)$ by $\wF$. 
For each $m \in \Z_{\ge 0}$, 
we put 
\[
\cI_m:=
\bigoplus_{\pp \in U} 
\mathrm{Map} (p^{-m}\Lambda/\Lambda ;\Z), 
\] 
where $\mathrm{Map} (p^{-m}\Lambda/\Lambda ;\Z)$ 
is the group of  $\Z$-valued functions on 
$p^{-m}\Lambda/\Lambda$,
and we identify $p^{-m}\Lambda/\Lambda$ with 
$\Gal (L_m/K) \simeq E[p^m](\C)$ 
via the natural isomorphism. 
We also define
\[
\cP_m :=
\left\{
(g_\pp)_\pp \in \cI_m
\mathrel{\bigg\vert} 
\sum_{\pp \in U} 
\sum_{\aa \in p^{-m}\Lambda/\Lambda} 
g_\pp (\aa) \cdot (p^{-m}z_\pp+\aa) \in 
\Z[1/p]\Lambda /\Lambda
\right\}.
\] 
Note that $\cP_m$ is a group 
independent of the choice of  the complex numbers
$z_\pp$. 
The second isomorphism of (\ref{eqisomell}) induces
$\cI(\OO \cap L_m) \simeq \cI_m$. 
The image of $\cP (\OO \cap L_m)$ in $\cI_m$
coincides with 
\(\cP_m \) 
since the isomoprhism 
$\Pic^0 (E_m) \simeq E_m(\C)$ 
implies that a fractional ideal $I$ of $\OO \cap L_m$ is principal 
if and only if the image of the divisor corresponding to $I$ 
in $E_m(\C)/E_m[p^{\infty}](\C)$ is equal to zero. 
We define 
\[
\cP:=
\varinjlim_m \cP_m \subseteq 
\varinjlim_m \cI_m =\Loc (T_p(E) ;\Z).
\]
Then, we obtain the isomorphisms 
$\cP \simeq \cP (\OO)$ and 
\[
\Cl (\OO) \simeq 
\left(
\bigoplus_{\pp \in U} 
\cSub (T_p(E) ;\Z)
\right)/ \cP. 
\]
\end{ex}

\begin{ex}[Maximal totally real abelian extension of $\Q$]\label{exab}
Put $K:=\Q$, and $\oo:=\Z$. 
Let $L$ be 
the maximal real abelian extension field of $\Q$,
namely $L:= \bigcup_{r \in \Q} \Q (\cos 2 \pi  r)$, 
and put $G:=\Gal (L/K) 
\simeq \widehat{\Z}^\times/\{ \pm 1 \}$. 
We denote the integral 
closure of $\oo$ in $L$ by $\OO$. 
Let $p$ be any prime number. 
By the class field theory, 
we have $D_{p}=\Z_p^\times \times \widehat{\Z}$, 
and 
\(
G/D_\pp \simeq \left( \prod_{\ell \ne p} \Z_\ell^\times 
\right)/\{ \pm 1 \} P_p
\), 
where $P_p \simeq \widehat{\Z}$ 
is the closure of the cyclic subgroup of
$\prod_{\ell \ne p} \Z_\ell^\times$ generated by $p$.
Note that $P_p$ is a torsion free group, 
but $\prod_{\ell \ne p} \Z_\ell^\times$ contains 
an infinite torsion subgroup 
$\bigoplus_{\ell \ne p} \F_\ell^\times$.
So, the set $G/D_\pp$ is an infinite.
(In particular, as mentioned in Remark \ref{remtopsp},
for each prime number $p$, 
the topological space 
$G/D_p$ is homeomorphic to the Cantor set.)
Note that 
by the local class field theory, 
we have $I_{p} \simeq \Z_p^\times$, and 
\(
v_{p\Z,L}(L^\times)= \varphi(2p)^{-1} \Z[1/p]
=(p-1)^{-1}\Z[1/p]
\), where $\varphi$ denotes the Euler function. 
So,  we have $\wV_{p\Z} =  \varphi(2p)^{-1} \Z[1/p] 
\sqcup \R$. 
In particular, 
we obtain $\Ram (\oo)=\mSpec \oo$. 
We put $X_p:= \prod_{\ell \ne p} \Z_\ell^\times$. 
Then, by definition, we obtain 
\begin{align*}
\ccSb (G/D_\pp;\wV_{p\Z}) &= \left\{ 
\wg \in \ccSb \left(
X_p ;\wV_{p\Z}
\right) \mathrel{\vert} \wg(\pm p x)=\wg(x) 
\ \text{for all $x \in X_p$}
\right\}, \\
\ccSr (G/D_\pp;\wV_{p\Z}) &= \left\{ 
\wg \in \ccSr \left(
X_p ;\wV_{p \Z}
\right) \mathrel{\vert} \wg(\pm p x)=\wg(x) 
\ \text{for all $x \in X_p$}
\right\}, \\
\ccSr (G/D_\pp;\val_{p\Z}) &= \left\{ 
\wg \in \ccSb \left(X_p ;\val_{p\Z}
\right) \mathrel{\bigg\vert} \wg(\pm p x)=\wg(x) 
\ \text{for all $x \in X_p$}
\right\}.
\end{align*}
In \cite{Ku}, Kurihara proved that $\Pic (\OO)=0$. 
Hence by Corollary \ref{corthmmain}, we obtain isomorphisms   
$\Cl (\OO) \simeq \bigoplus_{p} \cM_{p\Z}$ and 
$\Reg \Cl (\OO) \simeq \bigoplus_{p} \cN_{p\Z} $.
\end{ex}

\end{document}